\documentclass[aoas,preprint]{imsart}

\RequirePackage[OT1]{fontenc}
\RequirePackage{amsthm,amsmath}
\RequirePackage[numbers]{natbib}
\RequirePackage[colorlinks,citecolor=blue,urlcolor=blue]{hyperref}
\usepackage{algorithm}
\usepackage{algorithmic}
\usepackage{longtable}
\usepackage{bbm}
\usepackage{relsize}
\usepackage{enumerate}
\usepackage{bm}
\usepackage{amsfonts}
\usepackage{graphicx}

\startlocaldefs
\theoremstyle{plain}
\newtheorem{theorem}{Theorem}
\newtheorem{lemma}{Lemma}
\newtheorem{remark}{Remark}

\newcommand{\zap}[1]{}

\newcommand{\length}{\mathrm{length}}

\newcommand{\Rb}{\mathbb{R}}
\newcommand{\Eb}{\mathbb{E}}
\newcommand{\av}{\mathbf{a}}

\newcommand{\Nv}{\mathbf{N}}

\newcommand{\betav}{\pmb{\beta}}

\newcommand{\lambdav}{\pmb{\lambda}}

\newcommand{\ev}{\mathbf{e}}

\newcommand{\Pv}{\mathbf{P}}
\newcommand{\Fcal}{\mathbf{F}}
\newcommand{\allp}{\pmb{\pi}}
\newcommand{\allpset}{\mathbf{\Pi}}
\newcommand{\allstates}{\mathbb{S}^K}
\newcommand{\allstate}{\mathbf{s}}
\newcommand{\allstater}{\mathbf{S}}
\newcommand{\allactions}{\mathbb{A}^K}
\newcommand{\allaction}{\av}
\newcommand{\allar}{\mathbf{A}}
\newcommand{\allpr}{\mathbb{P}}
\newcommand{\allr}{R}

\newcommand{\subp}{\pi}
\newcommand{\subpset}{\Pi}
\newcommand{\subr}{r}
\newcommand{\substates}{\mathbb{S}}
\newcommand{\substater}{S}
\newcommand{\substate}{s}
\newcommand{\subactions}{\mathbb{A}}
\newcommand{\subar}{A}
\newcommand{\subpr}{P}
\newcommand{\subaction}{a}

\newcommand{\ind}[1]{\mathbbm{1}(#1)}

\endlocaldefs

\arxiv{arXiv:0000.0000}

\startlocaldefs
\numberwithin{equation}{section}
\theoremstyle{plain}

\endlocaldefs

\begin{document}

\begin{frontmatter}
\title{An Asymptotically Optimal Index Policy\\ for Finite-Horizon Restless Bandits}

\begin{aug}
\author{\fnms{Weici Hu} \ead[label=e1]{wh343@cornell.edu}},
\author{\fnms{Peter I. Frazier} \ead[label=e2]{pf98@cornell.edu}}
\affiliation{Operations Research and Information Engineering Department\\ Cornell University}

\address{Operations Research and Information Engineering department, Cornell University\\
\printead{e1}\\
\phantom{E-mail:\ }\printead*{e2}}

\end{aug}

\begin{abstract}
We consider restless multi-armed bandit (RMAB) with a finite horizon and multiple pulls per period. Leveraging the Lagrangian relaxation, we approximate the problem with a collection of single arm problems. We then propose an index-based policy that uses optimal solutions of the single arm problems to index individual arms, and offer a proof that it is asymptotically optimal as the number of arms tends to infinity. We also use simulation to show that this index-based policy performs better than the state-of-art heuristics in various problem settings.
\end{abstract}

\

\end{frontmatter}

\section{Introduction}\label{intro}
We consider the restless multiarmed bandit (RMAB) problem \citep{whittle1988} with a finite horizon and multiple pulls per period.  In the RMAB, we have a collection of ``arms'', each of which is endowed with a state that evolves independently.  If the arm is ``pulled'' or ``engaged' in a time period then it advances stochastically according to one transition kernel, and if not then it advances according to a different kernel.  Rewards are generated with each transition, and our goal is to maximize the expected total reward over a finite horizon, subject to a constraint on the number of arms pulled in each time period.

The RMAB generalizes the multi-armed bandit (MAB) \citep{robbins195x} by allowing arms that are not engaged to change state and multiple pulls per period. 
This extends the applicability of the MAB problem to a broader range of settings, including the submarine tracking problem, project assignment problem in \citep{whittle1988}, 
and contemporary applications including:
\begin{itemize}
\item Facebook displays ads in the \textit{suggested posts} section every time its users browse their personal pages. Among the ads that have been shown, some are known to attract more clicks than others. But there are also many ads which have yet to be shown and they may attract even more clicks. Given that the slots for display are limited, a policy is required to select ads to maximize total clicks.
\item In a multi-stage clinical trial, a medical group starts with a number of new treatments and an existing treatment with reliable performance. In each stage, a few treatments are selected from the pool to test, with the goal to identify the new treatments that perform better than the existing one with high confidence. A strategy is required to select which treatments to test at every stage to most effectively support their judgment at the end of the trial. 
\item A data analyst wishes to label a large number of images using crowdsourced effort from low-cost but potentially inaccurate workers. Each label given by the crowdworkers comes with a cost and the analyst has limited budget. Hence she needs to carefully assign tasks so as to maximize the likelihood of correct labeling.
\end{itemize}

The infinite horizon MAB with one pull per time period is famously known to have a tractable-to-compute optimal policy, called the Gittins index policy \citep{gittins1979}.  This policy is appealing because it can be computed by considering the state space for only a single arm, making it computationally tractable for problems with many arms.  This policy loses its optimality properties, however, when modifying the problem in any problem dimension: when allowing arms that are not engaged to change state; when moving to a finite horizon \citep{berry1985}; or when allowing multiple pulls per period.  Thus, the Gittins index does not apply to our problem setting.  Moreover, while optimal policies for RMABs with multiple pulls per period or finite horizons are characterized by the dynamic programming equations \citep{putermanBook}, the so-called `curse of dimensionality' \citep{PowellBook} prevents computing them because the dimension of the state space grows linearly with the number of arms.

Thus, while the RMAB is not known to have a computable optimal policy, \cite{whittle1988} proposed a heuristic called the Whittle index for the infinite-horizon RMAB with multiple pulls per period, which is well-defined when arms satisfy an indexability condition.  This policy is derived by considering a Lagrangian relaxtion of the RMAB in which the constraint on the number of arms pulled is replaced by a penalty paid for pulling an arm.  An arm's Whittle index is then the penalty that makes a rational player indifferent between pulling and not pulling that arm.  The Whittle index policy then pulls those arms with the highest Whittle indices.  Appealingly, the Whittle index and the Gittins index are identical when applied to the MAB problem with a single pull per period.  

\cite{whittle1988} further conjectured that if the number of arms and the number of pulls in each time period go to infinity at the same rate in an infinite-horizon RMAB, then the Whittle index policy is asymptotically optimal when arms are indexable. \cite{weber1990,weber1991} gave a proof to Whittle's conjecture with a difficult-to-verify condition: that the fluid approximation has a globally asymptotically stable equilibrium point.  This condition was shown to hold when each arm's state space has at most $3$ states, but this condition does not hold in general and \cite{weber1990} provides a counterexample with $4$ states.

Our contribution in this paper is to (1) create an index policy for finite horizon RMABs with multiple pulls per period, and (2) show that it is asymptotically optimal in the same limit considered by Whittle.  Like the Whittle index, our approach is computationally appealing because it requires considering the state space for only a single arm, and its computational complexity does not grow with the number of arms.  Unlike the Whitle index, our index policy does not require an indexability condition hold to be well-defined, and in contrast with \cite{weber1990,weber1991} our proof of asymptotic optimality holds regardless of the number of states.  We further demonstrate our index policy numerically on problems from the literature that can be formulated as finite-horizon RMABs, and show that it provides finite-sample performance that improves over the state-of-the-art. 

In addition to building on \cite{whittle1988,weber1990,weber1991}, our work builds on the literature in weakly coupled dynamic programs (WCDP), that itself builds on RMABs.
Indeed, at the end of his paper, Whittle pointed out that his relaxation technique can be applied to a more general class of problems in which sub-problems are linked by constraints on actions, but are otherwise independent. Hawkins in his thesis \citep{Hawkins2003} formally termed these problems (but with a more general type of constraints) as WCDPs and proposed a general decoupling technique. Moreover, he also proposed index-based policies for solving both finite and infinite horizon WCDPs and offered a proof that his policy, when applied to the infinite time horizon Multi-arm bandit problem (MAB), is equivalent to the Gittins index policy. Our work is similar to Hawkins' in that we consider Lagrange multipliers of the same functional form when computing indices. However, Hawkins does not specify what the coefficients of the function should be, or give a tie-breaking rule for the case when multiple arms have the same index. We obtain an asymptotically optimal policy by addressing both of these issues. The differences will be discussed with greater details after we formally introduce our index policy.

Another major work in WCDP is by \cite{adel2008} who shows that the ADP relaxation is tighter than the Lagrangian relaxation but is also computationally more expensive. It gives necessary and sufficient conditions for the Lagrangian relaxation to be tight and proves that the optimality gap is bounded by a constant when the Lagrange multipliers are allowed to be state dependent. The last result that the optimality gap is bounded by a constant implies that the per arm gap goes to zero as the number of arms grows. We achieve a similar result in our paper by showing the per arm reward of our index-based heuristic policy goes to the per arm reward of the Lagrangian bound, in spite of that our Lagrange multipliers are not state-dependent. We conjecture that this is due to the fact that our constraints, which is a function about on the action and not the state, is less general than the constraints considered in WCDP, which depends on both the action and the state. Moreover, the focuses of the two works differ: while our work focuses on offering an asymptotically optimal heuristic policy, \cite{adel2008} examines the ordering and tightness of different bounds. The heuristic proposed in \cite{adel2008} is based on ADP technique, is also different from our index-based policy.

Other work on WCDP also include \cite{Ye2014} who proposes a even tighter bound by incorporating information relaxation on the non-anticipative constraints in addition to the existing relaxation methods. \cite{Gocgun2011} considers two classes of large-scaled WCDPs in which the state and action space in each sub-problem also grows exponentially and uses an ADP technique to approximate the value functions of individual sub-MDPs in addition to employing Lagrangian relaxation for the overall problem. 

The remainder of this paper is outlined as follows: Section \ref{prob_desc} formulates the problem. Section \ref{sec:up} discusses the Lagrangian relaxation of the problem. Section \ref{sec:alg} states our index-based policy, and provide computation methods. Section \ref{sec:pf} gives a proof of asymptotic optimality. Section \ref{sec:num} numerically evaluates our index policy. Section 8 concludes the paper.

\section{Problem Description and Notation}\label{prob_desc}
We consider an MDP \\ $(\allstates,\allactions,\allpr^{\cdot},\allr)$ which is created by a collection of $K$ sub-processes $(\substates,\subactions,\subpr^{\cdot},\subr)$. The sub-processes are independent of each other except that they are bound by the joint decisions on their actions at each time step. These sub-processes are also referred to as {\it arms} in the bandit literature and shall be indexed by $x\in\{1,...,K\}$. 
Following a standard construction for MDPs, 
both the larger joint MDP and the sub-processes will be constructed on the same measurable space $(\Omega,\Fcal)$.
Random variables on this measurable space will correspond to states, actions, rewards,
and each policy will induce a probability measure over this space.

We describe the MDP to consider formally as follows:

\begin{itemize}
\item The \textbf{time horizon} $T<\infty$.

\item The \textbf{state space} $\allstates$ is the cross product of $K$ sub-process state space $\substates$. $\substates$ is assumed to be finite. We use $\allstate=(s_1,...,s_K)$ to denote an element in $\allstates$ and $\allstater$ when the state is random. We also use $\allstater_t$ to emphasize that the state is at time $t$. 
Likewise, we use $s$ to denote an element in $\substates$ and $S$ or $S_t$ when it is random. 

\item The \textbf{action space} $\allactions$ is the cross product of $K$ sub-processes action space $\subactions$, which is set equal to $\{0,1\}$. We use $\subaction$ to denote a generic element of $\subactions$, and $\subar$ when it is random. We use $\allaction=(\subaction_1,\subaction_2,...,\subaction_K)$ to denote a generic element in $\allactions$ and $\allar$ to denote an action when it is random. In the context of bandit problems, $\subaction=1$ is called ``pulling'' an arm (sub-process).
\item The \textbf{reward function} $\allr_t:\allstates \times \allactions\mapsto \mathbb{R}$ for each $1\leq t\leq T$. $\allr_t(\allstate,\allaction) = \sum_{x=1}^K \subr_t(\substate_x,\subaction_x)$, where  $\subr_t(\substate_x,\subaction_x)$ is the reward obtained by a sub-process when action $\subaction_x$ is taken in state $\substate_x$ at time $t$. We assume rewards are non-negative and finite. 

\item The \textbf{transition kernel} $\allpr^{\allaction}(\allstate',\allstate) = \prod_{x=1}^{K}\subpr^{\subaction_x}(\substate'_x,\substate_x)$, where $\subpr^{\subaction}(\substate',\substate)$ is the probability of a sub-process transitioning from $\substate'$ to $\substate$ if action $\subaction$ is taken, i.e., $P(\substate|\substate',\subaction)$. The product implies that the $K$ sub-processes evolve independently. We also point out that RMAB differ from MAB problems in that MABs require $\subpr^0(\substate,\substate)=1$ while RMABs allows $\subpr^0(\substate,\substate)<1$. Since we are considering both cases, we do not restrict the value of $\subpr^0(\substate,\substate)$.
\end{itemize}

 

Next we describe the set of policies for our MDP problem. Since the state and action space defined above are finite, it is sufficient to consider the set of Markov policies $\allpset$ \citep{putermanBook}. Define a policy $\allp\in\allpset$ as a function $\allstates\times \allactions \times \{1,...,T\} \rightarrow [0,1]$ that determines the probability of choosing action $\allaction$ in state $\allstate$ at time $t$, that is, $P(\allaction|\allstater_t=\allstate)$,  Subsequently we have $\sum_{\allaction\in\allactions}\allp(\allstate,\allaction,t)=1$, $\forall \allstate\in\allstates, \forall 1\leq t\leq T$. A policy $\allp$ and the transition kernel  $\allpr^{\cdot}(\cdot,\cdot)$ together defines a probability distribution $P^{\allp}$on the all possible paths of the process $\{\allstate_1\allaction_1...\allstate_T:\allstate_t\in \allstates, \allaction_t\in \allactions\}$. Starting at a fixed state $\allstate_1$, i.e., $P^{\allp}(\allstater_1=\allstate_1)=1$, we have the conditional distributions of $\allstater_t$ and $\allar_t$ defined recursively by $P^{\allp}(\allstater_{t+1}=\allstate'|\allstater_t=\allstate,\allar_t=\allaction)=\allpr^{\allaction}(\allstate,\allstate')$ and $P^{\allp}(\allar_t=\allaction|\allstater_t=\allstate)=\allp(\allstate,\allaction,t)$. 

The MDP we are considering allows exactly $m$ sub-processes to be set active at each time step. Hence a feasible policy, $\allp\in\allpset$, has to satisfy that $P^{\allp}(|\allar_t|=m)=1$, $\forall t\in\{1,...,T\}$. Here we use $|\cdot|$ as an operator that sums all the elements in a vector.  

The objective of our MDP is as follows:
\begin{equation}\label{prime}
\begin{aligned}
& \underset{\allp\in\allpset}{\text{maximize}}
& & \mathbb{E}^{\allp}\left[\sum_{t=1}^{T}\allr_t\big(\allstater_t,\allar_t\big)\right] \\
& \text{subject to}
& & P^{\allp}(|\allar_t|=m)=1, \hspace{3mm}\forall 1\leq t\leq T .
\end{aligned}
\end{equation}
Since we will discuss other MDPs in the process of solving this one,  \eqref{prime} will be referred to as the \textit{original} MDP in the rest of the paper to avoid confusion.  
For convenience, we summarizes our notations in Appendix \ref{ap:notations}. We note the original MDP \eqref{prime} suffers from the 'curse of dimensionality', and hence solving it is computationally intractable. In the remaining of the paper we seek to building a computationally feasible index-based heuristics with performance guarantee.




\section{Lagrangian Relaxation and Upper Bounds}\label{sec:up}
In this section we discuss the Lagrangian relaxation of the original MDP and the corresponding single process problems. These single process problems together with the Lagrange multipliers form the building block of our index-based policy,  which will be formally introduced in Section \ref{sec:alg}. Lagrangian relaxation removes the binding constraints and allows the problem to be decomposed into tractable MDPs \citep{adel2008}. It works as follows:
for any $\lambdav = \{\lambda_1,...,\lambda_T\} \in \Rb^{T}$, we consider an unconstrained problem whose objective is obtained by augmenting the objective in \eqref{prime}:
 \begin{equation}\label{ub}
 P(\lambdav)=\max_{\allp\in \allpset}\mathbb{E}^{\allp}\left[\sum_{t=1}^{T}\allr_t\big(\allstater_t,\allar_t\big)\right]-\mathbb{E}^{\allp}\left[\sum_{t=1}^T\lambda_t(|\allar_t|-m) \right].
 \end{equation} 
  This unconstrained problem forms the \textit{Lagrangian relaxation} of \eqref{prime}, and has the following property:
 \begin{lemma}\label{th:up}
For any $\lambdav\in\Rb^{T}$, $P(\lambdav)$ is an upper bound to the optimal value of the original MDP.
 \end{lemma}
 \cite{adel2008} gave a proof to Lemma \ref{th:up} using the Bellman equations. We provide a more straightforward proof by viewing $P(\lambdav)$ as the Lagrange dual function of a relaxed problem of the original MDP; see Appendix \ref{ap:up}. 
 
This Lagrangian relaxation then decomposes into $K$ smaller MDPs, which we can easily solve to optimality. To elaborate on this idea of decomposition, we construct a \textit{sub-MDP} problem based on tuple $(\substates,\subactions,\subpr^{\cdot}(\cdot,\cdot),\subr(\cdot,\cdot))$. Again we consider only the set of Markov policies, $\subpset$, for this problem. Similarly a policy $\subp\in\subpset$ is a function that determines the probability of choosing action $\subaction$ in state $\substate$ at time $t$, i.e., $\subp:\substates\times \subactions \times \{1,...,T\} \rightarrow [0,1]$. 
The sub-MDP starts at a fixed state $s_1$. Subsequently we can define distributions of $\substater_t$ and $\subar_t$ under $P^{\subp}$ in a similar manner as we did for $\allstater_t$ and $\allar_t$ in the previous section. The objective of the sub-MDP is defined as follows:
\begin{equation}\label{dpx}
Q(\lambdav)=\max_{\subp\in \subpset}\mathbb{E}^{\subp}\left[\sum_{t=1}^{T}\subr_t(\substater_t,\subar_t)-\lambda_t \subar_t\right].
\end{equation}

We are now ready to present the decomposition of the Lagrangian relaxation.
\begin{lemma}\label{th:decom}
The optimal value of the relaxed problem satisfies 
\begin{equation}\label{dec}
\Pv(\lambdav)=K Q(\lambdav) + m\sum_t\lambda_t,
\end{equation}
\end{lemma}
\cite{adel2008} also gave a proof to Lemma 2, and we again provide a different proof in Appendix \ref{ap:decom}.
Since the state space of the sub-MDP is much smaller, we can solve it directly by using backward induction on optimality equations. 
The existence of such an optimal Markov deterministic policy is implied by that the state and action spaces of the sub-MDP being finite \citep{putermanBook}. Let $\subpset^*(\lambdav)$ be the set of optimal Markov deterministic policies of the sub-MDP for a given $\lambdav$. The relaxed problem can be solved by combining the solutions of individual sub-MDPs, that is, we can construct an optimal policy of the relaxed problem $\allp^{\lambdav}$ by setting $\allp^{\lambdav}(\allstate,\allaction,t) = \prod_{x=1}^K\pi^{\lambdav}(\substate_x,\subaction_x,t)$, where $\subp^{\lambdav}$ is an element in $\subpset^*(\lambdav)$. 
Moreover, $\mathbf{P}(\lambdav)$ is convex and piecewise linear in $\lambdav$ \citep{adel2008}. 

\section{An Index Based Heuristic Policy}\label{sec:alg}
Our index based heuristic policy assigns an index to each sub-process, based upon its state and current time. At each time step, we set active the m sub-processes with the highest indices. Before carrying out the process of sequential decision-making, our index policy calls for pre-computation of 1) $\lambdav^*\in \arg\inf_{\lambdav}P(\lambdav)$, as defined in Section \ref{sec:up}; 2) a set of indices, $\betav$, that will later be used for decision-making at every time step; 3) an optimal policy $\subp^{**}$ for the sub-MDP problem in \eqref{dpx}. In the first part of this section we discuss how we carry out such computations. 
\subsection{Pre-computations}
\subsubsection{Dual optimal $\lambdav^*$}\label{subsec:pi}
We use subgradient descent to solve $\inf_{\lambdav} P(\lambdav)$, which converges to its solution $\lambdav^*$ by convexity of $\lambdav \mapsto P(\lambdav)$ (Theorem 7.4  in \citep{RusBook2006}).  By \eqref{dpx} and $\eqref{dec}$, a sub-gradient of $P(\lambdav)$ with respect to $\lambdav$ is given by $(-K\mathbb{E}^{\subp^{\lambdav}}[A_t]+m: 1\leq t\leq T)$, where $\subp^{\lambdav}$ is any policy in $\Pi^*(\lambdav)$.

To compute this sub-gradient, we compute a policy in $\Pi^*(\lambdav)$ and then use exact computation or simulation with a large number of replications to compute $\mathbb{E}^{\subp^{\lambdav}}[A_t]$.
To compute a policy in $\Pi^*(\lambdav)$, we first compute the value function $V^{\lambdav}:\substates\times\{1,...,T\}\mapsto \mathbb{R}$ of sub-MDP $Q(\lambdav)$.  We accomplish this using backward induction \citep{putermanBook}:
\begin{equation}
V^{\lambdav}(s,t)=
\begin{cases}
	\max_{\subaction\in \subactions}\{r_T(s,a)-a\lambda_T\}  & \text{if $t=T$,}\\
	\max_{\subaction\in\subactions}\{r_t(s,a)-a\lambda_t+\sum_{s'\in\substates}P^{a}(s,s')V^{\lambdav}(s',t+1)\}  & \text{otherwise.}
\end{cases} 
\end{equation}

Then, any and all policies $\pi^{\lambdav}$ in $\Pi^*(\lambdav)$ are constructed by determining for each $s$ and $t$ the action $a$ whose one-step lookahead value $r_t(s,a)-a\lambda_t+\sum_{s'\in\substates}P^{a}(s,s')V^{\lambdav}(s',t+1)$ is equal to $V^{\lambdav}(s,t)$, and then setting $\pi^{\lambdav}(s,a,t)=1$ for this $a$.
For those $s$ and $t$ for which both actions $a$ have one-step lookahead values equal to $V^{\lambdav}(s,t)$, one may set $\pi^{\lambdav}(s,a,t)=1$ for either such action.
Thus, the cardinality of $\Pi^*(\lambdav)$ is 2 raised to the power of the number of $s,t$ for which the one-step lookahead values for playing and not playing are tied.

When we construct a policy in $\Pi^*(\lambdav)$ for the purpose of computing a sub-gradient of $P(\lambdav)$, we choose to play in those $s,t$ with tied one-step lookahead values. 
While our subgradient descent algorithm would converge for other choices, making this choice better supports computation of indices in section~\ref{subsec:beta}, 

\subsubsection{Indices $\beta_t(s)$}
\label{subsec:beta}
Define vector $\mathbf{v}[a,t]$ to be $\mathbf{v}+(a-v_t)*\mathbf{e}_t$, that is, the vector $\mathbf{v}$ with the $t^{th}$ element replaced by $a\in\mathbb{R}$. 
We define the \textit{index} of state $\substate\in\substates$ at time $t$ as
\begin{equation}\label{df:index}
\beta_t(\substate) = \sup\{\beta: \exists \; \subp\in\subpset^*(\lambdav^*[\beta,t]) \;s.t.\; \subp(\substate,1,t)=1\}.
\end{equation} 
Instead of computing the entire set $\subpset^*(\lambdav^*[\beta,t])$, we only need to compute a policy in $\subpset^*(\lambdav^*[\beta,t])$ using the method discussed in section \ref{subsec:pi}, i.e., always choose the active action when there are ties. Intuitively, this index is the maximum price we are willing to pay to set a sub-process active in state $s$ at $t$.
By leveraging the monotonicity of optimal actions with respect to rewards, as shown in Lemma \ref{th:index} in Appendix \ref{ap:bisect}, we compute $\beta_t(s)$ via bisection search in interval $[0,U]$, where $U$ upper bounds the largest possible value of $\beta_t(s)$. For example, we can set $U$ as $T* \max_{s,a,t}\subr_t(s,a)$ when $\lambdav^*\geq 0$ (which we show in Appendix \ref{ap:upbd} that $\beta_t(s)$ cannot be greater than this value in this case). We pre-compute the set $\betav=\{\beta_t(\substate):\substate\in\substates,1\leq t\leq T\}$ before running the actual algorithm.

\subsubsection{Occupation measure $\rho$ and its corresponding optimal policy $\mathbf{\subp}^{\textbf{**}}$}
Our tie-breaking policy involves constructing an optimal Markov policy $\subp^{**}$ for the sub-MDP $Q(\lambdav^*)$ such that $\Eb^{\subp^{**}}[\subar_t]=\frac{m}{K}$, $\forall 1\leq t\leq T$. 
The existence of $\subp^{**}$ is shown in Appendix \ref{ap:exist}.
To compute $\subp^{**}$, we borrow the idea of occupation measure \citep{DynkinBook}. Define occupation measure, $\rho(s,a,t)$, induced by a policy $\subp$ to be the probability of being in state $s$ and taking action $a$ given time $t$ under $\subp$. Subsequently $\subp^{**}$ can be solved by the following linear program (LP):
\begin{equation}\label{eq:lp}
\begin{aligned}
& \underset{\{\rho(s,a,t):y\in\substates,a\in \subactions,t\in\{1,...,T\}\}}{\text{max}}
& & \sum_{t=1}^{T}\sum_{a\in\subactions}\sum_{s\in\substates}\rho(s,a,t)r'_t(s,a) \\
& \text{subject to}
& & \sum_{s\in\substates}\rho(s,1,t)=\frac{m}{K}, \; \forall t=1.\ldots,T\\
& & & \sum_{a\in \subactions}\rho(s,a,t) - \sum_{a\in\subactions}\sum_{s'\in\substates}\rho(s',a,t-1)\subpr^{a}(s',s)=0, \; \; \forall \substate\in\substates,2\leq t\leq T\\
& & & \sum_{a\in\subactions}\rho(s,a,1)=\ind{s=s_1} \;\;\forall s\in\substates\\
& & & \rho(s,a,t)\geq 0, \; \forall s\in\substates,a\in\subactions,t=1,\ldots,T,
\end{aligned}
\end{equation}
where $r'_t(s,a)=r_t(s,a)-\lambda^*_t\mathbbm{1}(a=1)$, $\forall s\in\substates, \subaction\in \subactions, 1\leq t\leq T$. The first constraint ensures that $\mathbb{E}^{\subp^{**}}[A_t]=\frac{m}{K}$. 
The second constraint ensures flow balance. 
The third constraint shows that we start at state $s_1$. 
The second and third constraint together imply $\sum_{a\in\subactions,s\in\substates} \rho(s,a,t)=1$, i.e., that $(\rho(s,a,t) : a\in\subactions,s\in\substates)$ is a probability distribution for each t.
The fourth and fifth constraints ensure that $\rho$ is a valid probability measure.

Let $\rho^{*}$ be an optimal solution to $\eqref{eq:lp}$, $\subp^{**}$ can then be constructed by
\begin{equation}\label{df:rho}
\subp^{**}(s,a,t)=
\begin{cases}
\frac{\rho^*(s,a,t)}{\sum_{a\in\subactions}\rho^*(s,a,t)},\; \; \text{if }\sum_{a\in\subactions}\rho^*(s,a,t)>0\\
\mathbbm{1}(a=1), \;\;\;\text{if }\sum_{a\in\subactions}\rho^*(s,a,t)=0 \text{ and }\beta_t(s) \geq \lambda^*_t\\
\mathbbm{1}(a=0), \;\;\;\text{if }\sum_{a\in\subactions}\rho^*(s,a,t)=0 \text{ and }\beta_t(s)< \lambda^*_t,
\end{cases}
\end{equation}for all $s\in \substates,\subaction\in\subactions,1\leq t\leq T.$

Here we also make an observation that $\lambdav^*\in P(\lambdav^*)$ can be computed by solving \eqref{eq:lp} with $r'$ replaced by $r$. 
\subsection{Index policy}
Let $\{\beta_x:x\in\{1,...,K\}\}$ be the indices associated with the K sub-processes at time $t$. 
We define $\bar{\beta}_t$ to be the largest value $\beta$ in $\{\beta_x:x\in\{1,...,K\}\}$ such that at least $m$ sub-processes have indices of at least $\beta$. Our index policy then sets the sub-processes with indices strictly greater than $\bar{\beta}_t$ active, and those with indices strictly less than $\bar{\beta}_t$ inactive. 
When more than $m$ sub-processes have indices greater than or equal to $\bar{\beta}_t$, a tie-breaking a rule is needed.
Our tie-breaking rule allocates remaining resources (the remaining sub-processes to be set active) across tied states according to the probability distribution induced by $\subp^{**}$ over $\substates$ at time $t$. 
This tie-breaking ensures asymptotic optimality of the index policy as it enforces that the fraction of sub-processes in each state $s$ is equal to the distribution induced by $\subp^{**}$ in the limit. 
This idea shall become clear in Section \ref{sec:pf} where the proof of asymptotic optimality is presented.

To further illustrate how our tie-breaking rule works, let $I_t= \{s_x:1\leq x\leq K,\beta_t(s_x)=\bar{\beta}_t\}$ be the set of states occupied by the tied sub-processes, we allocate \begin{equation}\label{tb_ratio}
\frac{\rho(s,1,t)}{\sum_{s'\in I_t}\rho(s',1,t)}
\end{equation} fraction of the remaining resources to each tied states, when $\sum_{s'\in I_t}\rho(s',1,t)>0$, where $\rho$ is a solution to \eqref{eq:lp}. 
In cases when $\sum_{s'\in I_t}\rho(s',1,t)=0$, we do tie-breaking according to the number of sub-processes that are currently in each of the tied states. 

We then use the function Rounding(total, frac, avail) in Algorithm \ref{ag:tiebreak} to deal with the situations where the products between the desired fractions and remaining resources are not integers. Here $total$ represents the number of remaining resources, $frac$ is a vector of fractions to be allocated to each tied state, and $avail$ is a vector of the number of sub-processes in each tied state. The function also takes care of the corner cases in which the number of sub-processes in a tied state $s$ is less than the number of resources we would like to assign to $s$ according to the fraction in \eqref{tb_ratio}. We note the following property of this function Rounding, which we will rely on in our proof in Section 6.

\begin{remark}
\label{remark:rounding}
When total, avail, frac satisfy $\text{avail}_i\geq \text{total} * \text{frac}_i$, the output vector $b=\text{Rounding}(\text{total},\text{frac},\text{avail})$ satisfies $|\text{b}_i-\text{total} * \text{frac}_i|<1$ for all $i$.
\end{remark}

We formally present our index policy in Algorithm \ref{algIndex} and \ref{ag:tiebreak}. 
\begin{algorithm}
\caption{Index Policy $\hat{\allp}$}
\label{algIndex}
\begin{algorithmic}
\STATE Pre-compute: $\lambdav^*$; $\betav$; $\rho$. (Refer to earlier discussions for computation details) 
  \FOR{$t = 1,...,T$}
    \STATE Let $\beta_{t,[i]}$ be the $i^{th}$ largest element in the list $\beta_t(\allstater_{t,1}),...,\beta_t(\allstater_{t,K})$, so $\beta_{t,[1]}\geq\ldots\geq \beta_{t,[K]}$. 
    \STATE Let $\bar{\beta}_t = \beta_{t,[m]}$
    \STATE Let $I_t=\{\substate : \beta_t(s)=\bar{\beta}_t\ \text{and}\ s=\allstater_{t,x}\ \text{for some x}\}$
    \STATE Let $N_t(s) = |\{x:\allstater_{t,x}=s\}|$, for all $\substate$.
    \STATE For $\substate\in I_t$, let 
    $$
    q(s) = 
    \begin{cases}
	\frac{\rho(s,1,t)}{\sum_{s'\in I_t}\rho(s',1,t)}, \; \text{if } \sum_{s'\in I_t}\rho(s',1,t)>0\\
	\frac{N_t(s)}{\sum_{s'\in I_t} N_t(s')}, \; \text{otherwise}
    \end{cases}
    $$
    Let $b = \mathrm{Rounding}(m-\sum_{s':\beta_t(s')>\bar{\beta}_t}N_t(s'),
    (q(s):s\in I_t),(N_t(s):s\in I_t))$
    \FOR{all $\substate$}
    \STATE If $\beta_t(s)>\bar{\beta}_t$, set all $N_t(s)$ sub-processes in $s$ active.
    \STATE If $\beta_t(s)= \bar{\beta}_t$, set $b(s)$ sub-processes in $s$ active.
    \STATE If $\beta_t(s)< \bar{\beta}_t$, set 0 sub-processes in $s$ active.
    \ENDFOR
  \ENDFOR
\end{algorithmic}
\end{algorithm}

\begin{algorithm}
\caption{Rounding(total, frac, avail)}
\label{ag:tiebreak}
\begin{algorithmic}
\STATE Inputs: total (a scalar), frac (a vector satisfying $\sum_i \text{frac}_i = 1$), avail (a vector of the same length as frac satisfying $\text{total} \le \sum_i \text{avail}_i$)
\STATE Output: $b$ (a vector of the same length as the inputs satisfying $\sum_i b_i = \text{total}$, $b_i \le \text{avail}_i$)
\STATE Let $n  =  \length(\text{frac})$
\STATE Let b$_i = \min\{\text{avail}_i, \lfloor \text{total}*\text{frac}_i\rfloor\}$, for $i = 1,...,n$.
\STATE Let $j = 1$
\WHILE{$\text{total} > \sum_{i=1}^{n}\text{b}_i$}
\STATE Let $\text{b}_{j} = \text{b}_{j} + \ind{\text{avail}_{j}> b_j}$
\STATE Let $j = (j \mod n) +1$ 
\ENDWHILE
\RETURN b
\end{algorithmic}
\end{algorithm}

\begin{remark} \cite{Hawkins2003} proposed a \textit{minimum-lambda} policy, which, when translated into our setting, finds the largest Lagrange multiplier $\lambdav$ of the form $\lambdav_0+r\lambdav_1$ for which an optimal solution of the relaxed problem is feasible for the original MDP. The policy then sets active those sub-processes which would be set active in the relaxed problem. However, Hawkins did not specify what $\lambdav_0$ and $\lambdav_1$ should be, thus limiting the policy's applicability to finite horizon settings. Our policy is similar to that of Hawkins in that 1) setting the sub-processes with the largest indices in our policy is equivalent to finding the largest $\lambdav$ that satisfies the constraints of the original MDP and setting the corresponding sub-processes active, and; 2) We also limit the values of Lagrange multiplier $\lambdav$ considered to a ray, as $\lambdav^*[\beta,t]$ can be written in the form of $\lambdav^* + r\ev_t$, for $r\in\Rb$. However, unlike Hawkins' policy, our policy defines the starting point and the direction of the ray, along with a tie-breaking policy that ensures asymptotic optimality. 
\end{remark}
\section{Proof of Asymptotic Optimality}\label{sec:pf}
Our index policy $\hat{\allp}$ achieves asymptotic optimality when we let the number of sub-processes $K$ go to infinity, while holding $\alpha = \frac{m}{K}$ constant.
Let $Z(\allp,m,K)$ to denote the expected reward of the original MDP obtained by policy $\allp$, to emphasize the dependency of this quantity on $K$ and $m$.
We use $\allpset_{m,K}$ to denote the set of all feasible Markov policies for the original MDP with $K$ sub-processes and a budget of $m$ activations per period. Lastly, it should be understood that whenever we use $\hat{\allp}$ to denote our index policy there is a dependency of $\hat{\allp}$ on $m$ and $K$ that is not explicitly stated. 
We are now ready to state the main result of this paper, which shows that the per arm gap between the upper bound and the index policy goes to zero under the limit assumption.:
\begin{theorem}\label{th:asp}
For any $\alpha\in(0,1)$, 
\begin{equation}
\lim_{K\rightarrow\infty}\frac{1}{K}\left(Z(\hat{\allp},\lfloor \alpha K\rfloor,K) - \max_{\allp\in\allpset_{\lfloor \alpha K\rfloor,K}}Z(\allp,\lfloor \alpha K\rfloor,K)\right) = 0.
\end{equation}
\end{theorem}

To formalize the notations that will be used throughout the proofs, we augment $P(\lambdav)$ to $P(\lambdav,m,K)$ to indicate the values of $m$ and $K$ assumed in the Lagrangian relaxation. We use $\lambdav^*$ to denote one and any element in $\arg\inf_{\lambdav}P(K,\alpha K,\lambdav)$ and let $\pi^{**}$ be the optimal policy constructed in \eqref{df:rho} using $m=\alpha K$, which satisfies $\mathbb{E}^{\pi^{**}}(A_t)=\alpha$. Note $\lambdav^*$ and $\pi^{**}$ depend on only $\alpha$ (not on $K$).


As before, we let $N_t(s)$ be the number of sub-processes in state $s$ at time $t$ under $\hat{\allp}$.  
We additionally define $M_t(\substate)$ to be the number of sub-processes in state $s$ at time $t$ that are set active by $\hat{\allp}$. These quantities depend on $K$ and $m$, but for simplicity we do not include this dependence in the notation: they always assume $m=\lfloor \alpha K \rfloor$ and we rely on context to make clear the value of $K$ assumed.
We also define $V_t(\substate)$ to be the set of states with the same index value as $s$, including $s$, and $U_t(s)$ to be the set of states with index value greater than that of $s$, for each time $t$.  These quantities depend on $\alpha$ but not on $K$ or $m$.

We prove Theorem \ref{th:asp} by first demonstrating below in Theorem~\ref{th:conv} that for each time $t$,
the proportion of the sub-processes that are in state $s$ under our index policy $\hat{\allp}$,
$\frac{N_t(s)}{K}$, approaches $P_t(s)$ as $K\rightarrow \infty$. 
In other words, our index policy $\hat{\subp}$ recreates the behavior of $\subp^{**}$ in the large $K$ limit.
\begin{theorem}\label{th:conv}
For every $s\in\substates$ and $1\leq t\leq T$,
\begin{equation}\label{conv1}
\lim_{K\rightarrow\infty}\frac{N_t(s)}{K} = P_t(s), \hspace{2mm} P^{\hat{\allp}}-a.s.,
\end{equation}
and
\begin{equation}\label{conv2}
\lim_{K\rightarrow\infty}\frac{M_t(s)}{K} = P_t(s)*\subp^{**}(s,1,t), \hspace{2mm} P^{\hat{\allp}}-a.s., 
\end{equation}
\end{theorem}
Before proving Theorem \ref{th:conv}, we first present two intermediate results, whose proofs are given in Appendix \ref{ap:eq} and \ref{ap:p1}.
\begin{lemma}\label{th:eq}
At time $1\leq t\leq T$, for all $\substate\in\substates$ , we have
\begin{enumerate}[(1)]
    \item If $\beta_t(\substate)>\lambda^*_t$, then $\subp^{**}(\substate,1,t) = 1$.
    \item If $\beta_t(\substate)<\lambda^*_t$, then $\subp^{**}(\substate,1,t) = 0$.
\end{enumerate}
\end{lemma}

\begin{lemma}\label{th:p1}
For any state $s\in\substates$ and time $1\leq t\leq T$, 
\begin{enumerate}[(1)]
\item If $\alpha-\sum_{s'\in U_{t}(s)\cup V_{t}(s)}P_{t}(s')\geq 0$, then $\subp^{**}(s,1,t) = 1.$
\item If $\alpha-\sum_{s'\in U_{t}(s)}P_{t}(s') \leq 0$, then $\subp^{**}(s,1,t) = 0.$
\end{enumerate}

\end{lemma}

We will also require the following technical result in the proof of Theorem~\ref{th:conv}. Again the proof is offered in Appendix \ref{ap:bino}

Now we are ready to prove Theorem \ref{th:conv}.
\begin{proof}{}
We prove \eqref{conv1} and \eqref{conv2} simultaneously via induction over the time periods.

When $t=1$, all sub-processes starts in state $s_1$, and we have
\begin{align*}
\lim_{K\rightarrow \infty}\frac{N_1(s)}{K} = \lim_{K\rightarrow \infty}\frac{K}{K} = 1 = P_1(s) \; \text{if }s = s_1,\\
\lim_{K\rightarrow \infty}\frac{N_1(s)}{K} = \lim_{K\rightarrow \infty}\frac{0}{K} = 0 = P_1(s) \; \text{otherwise}.
\end{align*}
By the set-up of the original MDP, $M_1(s)=\lfloor\alpha * K\rfloor$, and we have
\begin{align*}
&\lim_{K\rightarrow \infty}\frac{M_1(s)}{K} = \lim_{K\rightarrow \infty}\frac{\lfloor\alpha*K\rfloor}{K} = \alpha = \subp^{**}(s,1,t) =  P_1(s)*\subp^{**}(s,1,t) , \; \text{if }s = s_1,\\
&\lim_{K\rightarrow \infty}\frac{M_1(s)}{K} = \frac{0}{K} = 0 = P_1(s)*\subp^{**}(s,1,t), \; \text{otherwise},
\end{align*} 
so we have proved the base case of the induction. 

Now assume \eqref{conv1} and \eqref{conv2} hold up until time $t$. 
Fix a state $s\in \substates$ and time $1\leq t\leq T$, define $Y_t(s',s)$ to be the number of sub-processes set active by $\hat{\allp}$ in $s'$ at time $t$ which transition to state $s$ at time $t+1$, and $X_t(s',s)$ to be the number of sub-processes set inactive by $\hat{\allp}$ in $s'$ at time $t$ which transition to $s$ at time $t+1$. Note that $Y_t(s',s)$ and $X_t(s',s)$ also depend on $K$. 
We can subsequently express $N_{t+1}(s)$ as
\begin{align*}
&N_{t+1}(s) = \sum_{s'\in \substates} Y_t(s',s) + X_t(s',s).
\end{align*}
Dividing both sides by $K$, and taking $K$ to a limit, we get 
\begin{equation}\label{n1}
\lim_{K\rightarrow \infty}\frac{N_{t+1}(s)}{K} = \lim_{K\rightarrow \infty}\sum_{s'\in \substates} \frac{1}{K} Y_t(s',s) + \lim_{K\rightarrow \infty}\sum_{s'\in \substates}\frac{1}{K} X_t(s',s) 
\end{equation}
Note $Y_t(s',s)$ is a binomial random variable with $M_t(s')$ trials and success probability $\subpr^{1}(s',s)$
Similarly, $X_t(s',s)$ is a binomial random variable with $N_t(s')-M_t(s')$ trials and success probability $\subpr^{0}(s',s)$. 
We can rewrite the RHS of \eqref{n1} by applying Lemma \ref{th:bino}, which is stated at the end of the section:
\begin{align}
\lim_{K\rightarrow\infty}\frac{N_{t+1}(s)}{K} &= \sum_{s'\in\substates}\lim_{K\rightarrow\infty}\frac{M_t(s')}{K}*\subpr^1(s',s) + \sum_{s'\in\substates}\lim_{K\rightarrow\infty}\frac{N_t(s')-M_t(s')}{K}*\mathbb{P}_x^0(s',s)\nonumber\\
& = \sum_{s'\in \substates} P_t(s')*\subp^{**}(s',1,t+1)*\subpr^1(s',s)\\
&\;\;\;\;+\sum_{s'\in \substates}
P_t(s')(1-\subp^{**}(s',1,t+1))*\mathbb{P}_x^0(s',s)\hspace{2mm}a.s.\nonumber\\ 
&=P_{t+1}(s).\hspace{2mm}a.s. \label{ntlimit} 
\end{align}
The last equality follows as we have exhausted all the ways of getting to $s$ at time $t+1$. Hence we have shown (\ref{conv1}) holds for time $t+1$. 

To show (\ref{conv2}) holds for time $t+1$, define sets $\Pv_t = \{P_t(s):s\in\substates\}$, and $\Nv_t=\{N_t(s):s\in\substates\}$. We use notation $\frac{\Nv_{t}}{K}$ for the set which consists of all elements in $\Nv_t$ divided by $K$. Define function $f_s(\Nv_t,\lfloor\alpha K\rfloor)$ to represent the number of sub-processes set active at time $t$ in state $s$, that is,
\begin{align}
f_s(\Nv_t,\lfloor\alpha K\rfloor,H_1,H_2) =& \mathlarger{\mathlarger{\mathbbm{1}_{([\lfloor\alpha K\rfloor-\sum_{s'\in U_{t}(s)}N_{t}(s')]^+ \geq \sum_{s'\in V_{t}(s)}N_{t}(s'))}}}
*N_{t}(s)\nonumber\\
&+\mathlarger{\mathlarger{\mathbbm{1}_{([\lfloor\alpha K\rfloor-\sum_{s'\in U_{t}(s)}N_{t}(s')]^+ < \sum_{s'\in V_{t}(s)}N_{t}(s'))}}}*\nonumber\\
&\hspace{4mm}\mathlarger{\mathlarger{\mathbbm{1}_{(\lfloor\alpha K\rfloor-\sum_{s'\in U_{t}(s)}N_{t}(s')>0)}}}b_s(\Nv_t,\lfloor\alpha K\rfloor), 
\end{align}
where $b_s(\Nv_t,\lfloor\alpha K\rfloor)$ represent the number of sub-processes set active when tie-breaking is needed, that is,
\begin{align}
b_s(\Nv_t,\lfloor\alpha K\rfloor) = & \mathlarger{\mathlarger{\mathbbm{1}_{(\sum_{s'\in V_{t}(s)}\rho(s',1,t)>0)}}} * \nonumber\\
&\hspace{4mm}\left(\min\Big\{\Big\lfloor (\lfloor\alpha K\rfloor-\sum_{s'\in U_{t}(s)}N_{t}(s'))\frac{\rho(s,1,t)}{\sum_{s'\in V_{t}(s)}\rho(s',1,t)}\Big\rfloor,N_{t}(s)\Big\}+H_1\right)\nonumber\\
& + \mathlarger{\mathlarger{\mathbbm{1}_{(\sum_{s'\in V_{t}(s)}\rho(s',1,t)=0)}}}\left(\Big\lfloor (\lfloor\alpha K\rfloor-\sum_{s'\in U_{t}(s)}N_{t}(s'))\frac{N_{t}(s)}{\sum_{s'\in V_{t}(s)}N_{t}(s')}\Big\rfloor+H_2\right),
\end{align}
where $H_1$ and $H_2$ are random variables due to the rounding rules in Algorithm \ref{ag:tiebreak}, and are dependent on $K$.
We also define function
\begin{equation}
g_s(\Pv_t) = 
\begin{cases}
\min\{P_t(s),[\alpha - \sum_{s'\in U_t(s)}P_t(s')]^+\frac{\rho(s,1,t)}{\sum_{s'\in V_{t}(s)}\rho(s',1,t)}\} \;\;\; \text{if } \sum_{s'\in V_{t}(s)}\rho(s',1,t) >0 \\
\min\{P_t(s),[\alpha - \sum_{s'\in U_t(s)}P_t(s')]^+\frac{P_t(s)}{\sum_{s'\in V_{t}(s)}P_t(s')}\} \;\;\; \text{if } \sum_{s'\in V_{t}(s)}\rho(s',1,t) =0 ,
\end{cases}
\end{equation}
This proof will be accomplished by the following three lemmas, whose proof is given in Appendix \ref{ap:movein},\ref{ap:equiv},\ref{ap:ppi}
\begin{lemma}\label{th:movein}
$$\lim_{K\rightarrow\infty}f_s(\frac{\Nv_{t+1}}{K},\frac{\lfloor\alpha K\rfloor}{K}) = f_s(\Pv_{t+1},\alpha), a.s$$.
\end{lemma}
\begin{lemma}\label{th:equiv}
$$f_s(\Pv_t,\alpha) = g_s(\Pv_t)$$
\end{lemma}
\begin{lemma}\label{th:ppi}
$$g_s(\Pv_t) = P_t(s)\pi^{**}(s,1,t)$$
\end{lemma}
Combining the three lemmas above we have  $$\lim_{K\rightarrow\infty}\frac{M_{t+1}(s)}{K}=\lim_{K\rightarrow\infty}f_s(\frac{\Nv_{t+1}}{K},\frac{\lfloor\alpha K\rfloor}{K}) =g_s(\Pv_{t+1})=P_{t+1}(s)\pi^{**}(s,1,t+1).$$ 
$\square$
\end{proof}

Finally, we prove Theorem \ref{th:asp} by leveraging the results from Theorem \ref{th:conv}.
\begin{proof}{Proof of Theorem \ref{th:asp}}
$\hat{\allp} \in \allpset_{\lfloor \alpha K\rfloor,K}$ implies $Z(\hat{\allp},\lfloor \alpha K\rfloor,K) \leq \max_{\allp\in\allpset_{\lfloor \alpha K\rfloor,K}}Z(\allp,\lfloor \alpha K\rfloor,K)$.  Thus,
\begin{equation*}
\lim_{K\rightarrow\infty}\frac{1}{K}Z(\hat{\allp},\lfloor \alpha K\rfloor,K) \leq  \lim_{K\rightarrow\infty}\frac{1}{K}\sup_{\allp\in\allpset_{\lfloor \alpha K\rfloor,K}}Z(\allp,\lfloor \alpha K\rfloor,K).
\end{equation*}
On the other hand,
\begin{align*}
\lim_{K\rightarrow\infty}\frac{1}{K}Z(\hat{\allp},\lfloor \alpha K\rfloor,K) 
=&\lim_{K\rightarrow\infty}\frac{1}{K}\Eb^{\hat{\allp}}\left[\sum_{t=1}^{T}\sum_{s\in\substates}r_t(s,1) M_{t}(s)+r_t(s,0) (N_{t}(s)-M_{t}(s))\right]\\
=&\sum_{t=1}^{T}\sum_{s\in\substates}r_t(s,1) \lim_{K\rightarrow\infty}\frac{1}{K}\Eb^{\hat{\allp}}\left[M_{t}(s)\right]+r_t(s,0) \lim_{K\rightarrow\infty}\frac{1}{K}\Eb^{\hat{\allp}}\left[N_{t}(s)-M_{t}(s)\right]\\
=&\sum_{t=1}^{T}\sum_{s\in \substates}\left[r_t(s,1) \rho(s,1,t)+r_t(s,0) \rho(s,0,t) \right]\\
=&\sum_{t=1}^{T}\sum_{s\in \substates}\left[r_t(s,1) \rho(s,1,t)+r_t(s,0) \rho(s,0,t) \right] - \mathbb{E}^{\subp^{**}}\left[\sum_{t}\lambda_t \left(A_t-\alpha\right)\right]\\
=& Q(\lambdav^*)+\alpha \sum\lambda^*_t\\
=& \lim_{K\rightarrow\infty}\frac{1}{K}(KQ(\lambdav^*)+\lfloor \alpha K \rfloor\sum\lambda^*_t)\\
=& \lim_{K\rightarrow\infty}\frac{1}{K} P(\lambdav^*,\lfloor \alpha K \rfloor, K)\\
\geq& \lim_{K\rightarrow\infty}\frac{1}{K}\sup_{\allp\in\allpset_{\lfloor \alpha K\rfloor,K}}Z(\allp,\lfloor \alpha K\rfloor,K).
\end{align*}
Here, the third line follows by Theorem \ref{th:conv} and the fact that both $N_t(s)$ and $M_t(s)$ are bounded and hence uniformly integrable random variables (for uniformly integrable random variables, convergence almost surely implies convergence in expectation). The fourth line holds because $\subp^{**}$ takes the active action at each time with probability $\alpha$. 
The fifth line follows from Lemma~\ref{th:decom}, where we have augmented the notation for $P$ to include the values of $m$ and $K$ assumed.
The sixth line follows from Lemma~\ref{th:up}.

Finally, sandwiching the two inequalities gives the desired result.
$\square$
\end{proof}

\section{Numerical Experiments}\label{sec:num}
In this section we present numerical experiments for two problems: the finite-horizon multi-arm bandit with multiple pulls per period,and subset selection \citep{chen2008,Law1985}. These experiments demonstrate numerically that our index policy is indeed asymptotically optimal. We also compare the finite-time performance of our policy to other policies from the literature.  
Although our previously provided theoretical results do not apply to finite $K$, we see that our index policy performs strictly better than all benchmarks considered in both of the problems.

\subsection{Multi-armed bandit}\label{subsec:mab}
In our first experiment, we consider a Bernoulli multi-armed bandit problem with a finite time horizon $T=6$, and multiple pulls per time period. A player is presented with $K$ arms and may select $m=\lfloor K/3\rfloor$ of them to pull at every time st. Each arm pulled returns a reward of $0$ or $1$. 
The player's goal is to maximize her total expected reward. We take a Bayesian-optimal approach and impose a Beta(1,1) prior on each of the arm. The values of the state then correspond to the posterior parameters of the K arms.

For comparison, we include results from an upper confidence bound (UCB) algorithm with pre-trained confidence width. At every time step, we compute $\mu_i + \alpha*\delta_i$ for each arm $i$, where $mu_i$ and $\delta_i$ are the sample mean and standard deviation of arm $i$. We pre-train $\alpha$ by running the UCB algorithm on a different set of data (but simulated with the same distribution) with values of $\alpha$ ranging from 0 to 5 and then set $\alpha$ to the value that gives the best performance. 

Figure \ref{fig:mab} plots the reward per arm (expected total reward divided by $K$) 
against $K$, for $K=12,120,1200,12000$. 
The red dashed line represents the upper bound computed using $P(\lambdav^*)$. For each policy, circles show the sample mean of the total reward per arm, and vertical bars indicate a 95\% confidence interval for the expected total reward per arm.  The UCB policy's sample means are connected by a dashed green line, and the index policy's are connected by a dashed black line.
Values are calculated using 5000 replications.

The index policy consistently outperforms the UCB policy. As $K$ grows large, the confidence interval for the index policy's total performance per arm overlaps with the upper bound, which numerically attests to the accuracy of Theorem \ref{th:asp} and illustrates the rate of convergence.

\begin{figure}\label{fig:mab}
\begin{center}
\includegraphics[scale=0.3]{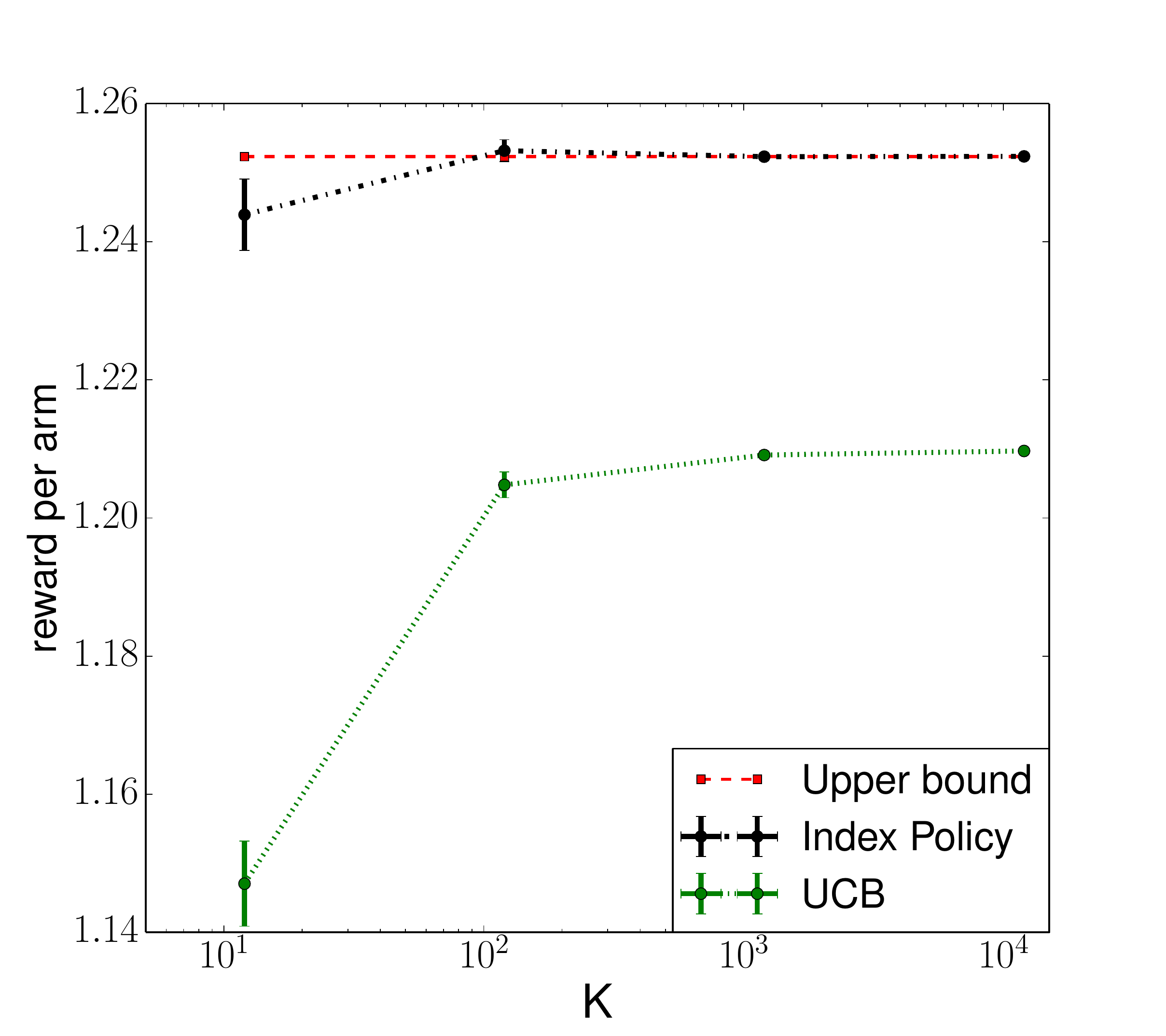}
\caption{upper bound and simulation results of MAB}
\end{center}
\end{figure}

\subsection{Subset selection problem}
In the third experiment, we consider a subset selection problem in ranking and selection whose goal is to identify $m$ best designs out of $K$ designs, each with some underlying distribution $\theta_x$. This problem is considered in \citep{chen2008} as well as \citep{Law1985}. We assume $\bar{m}$ parallel computing resources are available; at each time step we select $\bar{m}$ out of $K$ design to evaluate. After $T$ rounds of evaluation, we select $m$ best designs. In this numerical study, we set $T=4$, $\frac{m}{K}=0.3$ and $\frac{\bar{m}}{K}=0.5$. We consider the situation when the outcomes of evaluation are binary. But note that our model can handle any real-valued outcomes.

Below is how we formulate this problem as an RMAB:
\begin{equation}\label{prime_ss}
\begin{aligned}
& \underset{\allp\in\allpset}{\text{maximize}}
& & \mathbb{E}^{\allp}\left[\sum_{t=1}^{T+1}\allr_t\big(\allstater_t,\allar_t\big)\right] \\
& \text{subject to}
& & P^{\allp}(|\allar_t|=m)=1, \; \; \text{for }t=T+1,\\
& & & P^{\allp}(|\allar_t|=\bar{m})=1, \; \; \text{for }1\leq t\leq T,
\end{aligned}
\end{equation}
where $R_t(\allstater_t,\allar_t)=0$ when $t\leq T$, and $R_t(\allstater_t,\allar_t)=\sum_{x=1}^{K}\mathbb{E}[\theta_x|\allstater_{t,x}]$ when $t= T+1$. We start with a uniform prior for each design. Note that in this formulation, the number of sub-processes allowed to set active varies with time horizon. Although this number, denoted as $m$, is fixed in Theorem \ref{th:asp}, we can show that the result still holds for a time dependent $m_t$.

We compare the performance of our policy against the \textit{OCBA-m} selection procedure proposed in \citep{chen2008}. Since \citep{chen2008} considers a slightly different setting in which a policy maker can evaluate a design more than once in a time step, we modify the procedure slightly to fit our setup: instead of sampling according to the number of times dictated by the algorithm, we rank the designs by their desired number of samples, and simulate the first $\bar{m}$ of them.  Moreover, since OCBA-m begins with a cold-start, for fair comparison, we allocate a sample corresponding to a positive outcome and a sample corresponding to a negative outcome to each of the design in addition to the total $T\bar{m}$ samples (Recall for the index policy we start with a uniform prior). We also use the UCB policy as a standard of comparison. The implementation of the UCB policy is similar to the one in section \ref{subsec:mab}.

The simulation results show that all the three policies perform closely when $K$ is small, with OCBA-m policy having a slight edge for $K=10$. From $K=100$ onwards, the index policy consistently outperforms the other two. In addition, the gap between the upper bound and the index policy vanishes as K becomes large, while the gaps between the upper bound and the other two policies remain constant.
\begin{figure}\label{fig:os}
\centering
\includegraphics[scale=0.3]{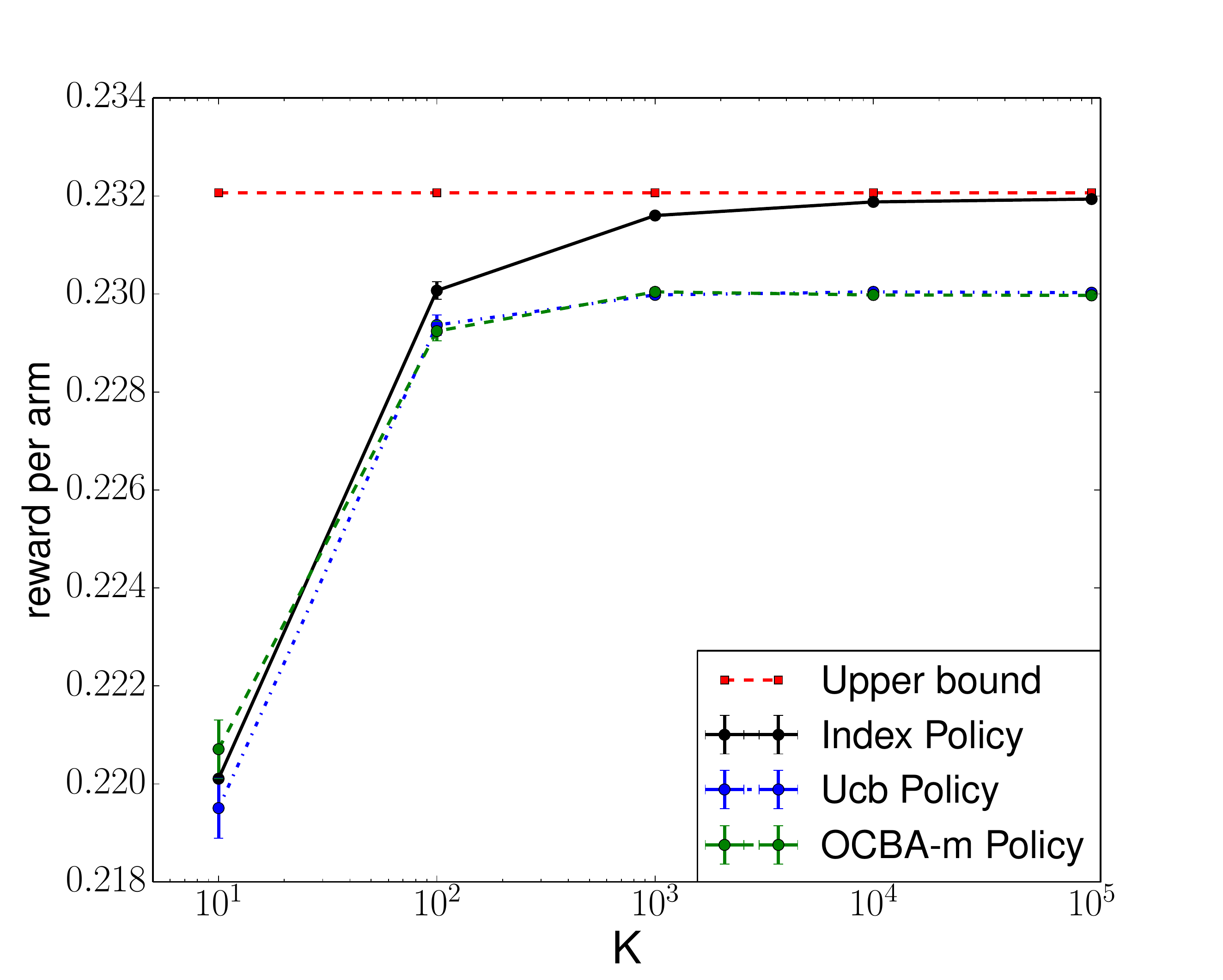}
\caption{Upper bound and simulation result of subset selection}
\end{figure}

\section{Conclusion}
In this paper we propose an index-based policy for finite horizon RMAB that is computational tractable, and prove that it is asymptotically optimal in the same limit as considered by Whittle. We also show that the numerical performance of this index-based policy beats the state-of-art. For future work, we conjecture that our results, including the formulation of the policy and the asymptotic optimality, can be extended to the following situations:
\begin{enumerate}
\item Multiple actions associated with a state, instead of an  active and a passive action in the current formulation;
\item A total budget constraint over the entire time horizon, in addition to budget constraint at every time step.
\item Infinite state space.
\end{enumerate}

\appendix
\section{Notation}\label{ap:notations}
\begin{centering}
\begin{longtable}{| p{.25\textwidth} | p{.75\textwidth} |}
\hline
$\allstates,\allactions,\allpr^{\cdot}(\cdot|\cdot),\allr(\cdot,\cdot)$  & State space, action space, transition kernal and reward function of the original MDP.\\ 
\hline
$\substates,\subactions,\subpr^{\cdot}(\cdot,\cdot),\subr(\cdot,\cdot)$& State space, action space, transition kernal and reward function of the sub-processes of the original MDP.\\ 
\hline
$\allstate$, $\allstater$ & Generic element and random element of $\allstates$.\\ 
\hline
$\substate$, $\substater$ & Generic element and random element of $\substates$. \\
\hline
$K$ & Number of sub-processes.\\ 
\hline
$T$& Time horizon\\ 
\hline
$m$& Number of sub-processes to be set active per time step\\ 
\hline
$\allpset$& Set of all Markov policies of the original MDP\\ 
\hline
$\allp(\allstate,\allaction,t)$ & the probability of choosing action $\allaction$ in state $\allstate$ under policy $\allp$ at time $t$.\\
\hline
$\subpset$& Set of all Markov policies of the sub-MDP\\ 
\hline
$\subp(\substate,\subaction,t)$ & the probability of choosing action $\subaction$ in state $\substate$ under policy $\subp$ at time $t$.\\
\hline
$\subpset^*(\lambdav)$& Set of Markov deterministic optimal policies for sub-MDP $Q(\lambdav)$, given $\lambdav\in\Rb^{T}$.\\
\hline
$\subp^{\lambdav}$& An element in $\subpset^{\lambdav}$.\\
\hline
$\allp^{\lambdav}$& A deterministic optimal policy for the relaxed problem which obtained by the decomposition method in Lemma \ref{th:decom}, given $\lambdav\in\Rb^{T}$.\\
\hline
$\Pv(\lambdav)$& Optimal value of the relaxed problem, given $\lambdav\in\Rb^{T}$.\\ 
\hline
$\lambdav^*$& An value that attains $\inf_{\lambdav}\mathbf{P}(\lambdav)$\\ 
\hline
$\subp^{**}$& An optimal markov policy for the sub-MDP which satisfies $\Eb^{\subp^{**}}[\subar_t]=\frac{m}{K}$, $\forall 1\leq t\leq T $.\\ 
\hline
$\hat{\allp}$& The index based policy proposed by this paper\\ 
\hline
$\bar{\beta_t}$& Indices of the tied sub-processes.\\ 
\hline
$I_t$& The set of states occupied by the tied sub-processes.\\ 
\hline
$N_t(s)$& The number of sub-processes in state $s$ at time $t$ under index policy $\hat{\allp}$.\\ 
\hline
$P_t(s)$& The probability of an individual sub-process landing in state $s$ at time $t$ under $\subp^{**}$. $P_t(s)=P^{\subp^{**}}[S_t=s]$\\ 
\hline
$\alpha$& The ratio between the number of sub-processes set active, $m$, and the total number of sub-processes $K$.\\ 
\hline
$M_t(s)$& The number of sub-processes in state $s$ at time $t$ that are set active under our index policy $\hat{\allp}$.\\ 
\hline
$Y_t(s',s)$ & The number of sub-processes set active by $\hat{\allp}$ in $s'$ at time $t$ which transition to state $s$ at time $t+1$. \\
\hline
$X_t(s',s)$ & The number of sub-processes set inactive by $\hat{\allp}$ in $s'$ at time $t$ which transition to $s$ at time $t+1$.\\
\hline
$U_t(s)$ & The set of states whose indices are greater than the index of state $s$ at time $t$. $U_t(s) = \{s''\in \substates:\beta_{t}(s'')>\beta_{t}(s)\}$. \\
\hline
$V_t(s)$ & The set of states whose indices are equal to that of $s$.

 $V_t(s) = \{s''\in\substates:\beta_t(s'')=\beta_t(s)\}$.\\
\hline
$|\mathbf{v}|$ &  An operation that sums all the elements in vector $\mathbf{v}$.  \\
\hline
$H_1$, $H_2$ & random variables due to the rounding rules in Algorithm \ref{ag:tiebreak}\\
\hline
$Z(\allp,m,K)$ & the expected reward of the original MDP obtained by policy $\allp$\\
\hline
\caption{List of notation}
\label{notations}
\end{longtable}
\end{centering}
\section{Upper Bound}\label{ap:up}
\proof{Proof of Lemma \ref{th:up}}
Let $\allpset_P = \{\allp\in\allpset: P^{\allp}(|\allar_t|=m)=1,\;\forall 1\leq t\leq T\}$. Let $\allpset_E = \{\allp\in\allpset:\mathbb{E}^{\allp}[|\allar_t|]=m,\;\forall 1\leq t\leq T\}$. For any $\lambdav\in\mathbb{R}^{T}$, we have
\begin{align*}
& P(\lambdav)\\
 =& \max_{\allp\in \allpset}\mathbb{E}^{\allp}\left[\sum_{t=1}^{T}\allr_t\big(\allstater_t,\allar_t\big)\right]-\mathbb{E}^{\allp}\left[\sum_{t}\lambda_t(|\allar_t|-m) \right] \\
\geq & \max_{\allp\in \allpset_E}\mathbb{E}^{\allp}\left[\sum_{t=1}^{T}\allr_t\big(\allstater_t,\allar_t\big)\right]-\mathbb{E}^{\allp}\left[\sum_{t}\lambda_t(|\allar_t|-m) \right] \\
= & \max_{\allp\in \allpset_E}\mathbb{E}^{\allp}\left[\sum_{t=1}^{T}\allr_t\big(\allstater_t,\allar_t\big)\right]\\
\geq & \max_{\allp\in \allpset_P}\mathbb{E}^{\allp}\left[\sum_{t=1}^{T}\allr_t\big(\allstater_t,\allar_t\big)\right],
\end{align*}
which is the optimal value of the original MDP. The first inequality is due to $\allpset_E\subseteq\allpset$. The first equality is due to the fact that any policy $\allp$ in $\allpset_E$ satisfies $\mathbb{E}^{\allp}[\allar_t|]=m$. The last inequality is due to $\allpset_P\subseteq\allpset_E$.
\endproof
\section{Decomposition}\label{ap:decom}
\proof{Proof of Lemma \ref{th:decom}}: 
\begin{align*}
 & \max_{\allp\in \allpset}\mathbb{E}^{\allp}\left[\sum_{t=1}^{T}\allr_t\big(\allstater_t,\allar_t\big)\right]-\mathbb{E}^{\allp}\left[\sum_{t=1}^T\lambda_t\big(|\allar_t|-m\big)\right] \\
 =& \max_{\allp\in \allpset}\mathbb{E}^{\allp}\left[\sum_{t=1}^{T}\allr_t\big(\allstater_t,\allar_t\big)-\lambda_t|\allar_t|\right] + m\sum_{t=1}^{T}\lambda_t\\
  =& \max_{\allp\in \allpset}\mathbb{E}^{\allp}\left[\sum_{t=1}^{T}\sum_{x=1}^{K}r_t(\substater_{t,x},\subar_{t,x})-\lambda_t\subar_{t,x}\right] + m\sum_{t=1}^{T}\lambda_t \\
 =&\sum_{x=1}^{K}\max_{\subp\in \subpset}\mathbb{E}^{\subp}\left[\sum_{t=1}^{T}r_t(\substater_{t},\subar_{t})-\lambda_t\subar_{t}\right] + m\sum_{t=1}^{T}\lambda_t \\
\end{align*}
\endproof
The first equality is due to linearity of expectation. The second equality is obtained by the definition of $r_t(\cdot,\cdot)$ and $|\cdot|$. The third equality is obtained by the independence of the process under policies in $\allpset$. 
\section{Show $\arg\inf_{\lambdav\in \mathbb{R}^T}\mathbf{P}(\lambdav)$ is non-empty}\label{ap:nonemp}
\begin{proof}
When $\lambdav \geq \mathbf{0}$, $\mathbf{P}(\lambdav)=\sum_x R_x(\lambdav) + m\sum_t\lambda_t \geq 0+0 = 0$. $R_x(\lambdav)$ is bounded below by 0 since a policy of not playing at all gives a total reward of 0. When $\lambdav < \mathbf{0}$, the cost of playing is negative, an optimal policy will always play at all time steps. Hence $\mathbf{P}(\lambdav)\geq m\left(\sum_t(0-\lambda_t)\right)+m\sum_t \lambda_t=0$. For the case in which $\lambdav$ contains both positive and negative entries, writing $\lambdav$ as a convex combination of $\lambdav_1 > 0$ and $\lambdav_2 < 0$ and we have that $P(\lambdav)$ is still bounded below by zero, since $P(\lambdav)$ is convex in $\lambdav$. Hence we can conclude that $\inf_{\lambdav\in \mathbb{R}^T}\mathbb{P}(\lambdav)$ exists (note here we make no claim about whether this infimum is attained by any finite $\lambdav$) and denote this value by $h^*$.

Recall we have assumed in the setup that all the rewards are bounded and non-negative, let $\bar{r}$ be an upper bound for all the reward values. For any $\lambdav$ with $\lambda_t\geq T\bar{r} $, the corresponding optimal policies for a single-arm problem will be not play at time t, for $T\bar{r}$ is at least the maximum reward obtainable by the single-arm problem. Hence $\mathbf{P}(\lambdav) \geq 0 +  mT\bar{r}$. For any $\lambdav\geq \mathbf{0}$, $\mathbf{P} = m\mathbb{E}[\sum_t r_{t,x}(S_{t,x},1)-\lambda_t|s_{1,x}]+m\sum_t \lambda_t=m\mathbb{E}[\sum_t r_{t,x}(S_{t,x},1)|s_{1,x}]$, which is independent of $\lambdav$. Hence the infimum is attained on the set $H=\{\lambdav: \lambda_t\geq \forall t \text{ and } \max_t \lambda_t\leq t\bar{r}\}$. Since $H$ is compact, there exists a $\lambdav^* \in H$ s.t. $\mathbf{P}(\lambdav^*)=h^*$. Hence $\arg\inf_{\lambdav\in \mathbb{R}^T}\mathbf{P}(\lambdav)$ is non-empty.
\end{proof}
\section{Proof the existence of $\pi^{**}$}\label{ap:exist}
The proof uses Theorem 3.6 in \cite{AltmanBook}. 
The setup in \cite{AltmanBook} is different from our problem in the following ways:
\begin{itemize}
    \item it deals with an infinite horizon problem, while we have a finite horizon problem.
    \item it has a discount parameter $\beta$ such that $0<\beta<1$, while we do not have any discount.
    \item the constraint of the original constrained problem is in the form of an inequality, while our constraints are equalities.
\end{itemize}
To be able to apply Theorem 3.6 to our problem, we need to consolidate the differences. Here is how we transform our problem:
\begin{itemize}
    \item To transform our problem to a problem with infinite horizon, we add an absorbing state $q$ such $r(q,\av) = 0$ and $\mathbb{P}(q|\allstate_T,\av)=1$ for all $\av$.
    \item We can add a discount parameter $\beta\in(0,1)$ and multiply each reward at time $t$ by $\frac{1}{\beta^{t}}$, and the value of the original problem stays the same.
    \item \cite{AltmanBook} only uses the fact that $\{\lambdav:\lambdav \geq 0\}$ is convex, and so is $\{\lambdav: \text{no constraints}\}$, so no transformation needed.
\end{itemize}
Apply Theorem 3.6, we get, there exists a $\subp^{**}\in {\subpset}_{M}$ such that
\begin{equation}\label{switch}
Q(\lambdav^*)=\inf_{\lambdav}\sup_{\subp\in \subpset_{D}} Q(\lambdav, \subp)= \sup_{\subp\in \subpset}\inf_{\lambdav} Q(\lambdav,\subp) = \inf_{\lambdav}Q(\lambdav, \subp^{**}).
\end{equation}
Since $\subp^{**}$ attains $Q(\lambdav^*)$, it is optimal. Moreover, it has to satisfy $\mathbb{E}^{\subp^{**}}\left[\sum_{t}\lambda_t*(\subar_t-\frac{m}{K})\Big | \substate_1\right]= 0$, for otherwise there is incentive for $\lambdav$ to go to either positive or negative infinity to attain the infimum. However we know $\mathbf{P}_x(\lambdav^*)$ has finite values since each reward is finite, that forces $\mathbb{E}^{\subp^{**}}[\sum_x a_{t,x}-\frac{m}{K}|s_0]= 0$ for every $t$. 

\section{Proof of  $T*$\lowercase{$\max_{s,a,t}\subr_t(s,a)$} upper bounds \lowercase{$\beta_t(s)$}}\label{ap:upbd}
It is sufficient to show that for any $\lambdav$ with $\lambda_t > T*\max_{s,a,t}\subr_t(s,a)$, $V^{\lambdav}(s,t)$ is attained by choosing $a=0$. When $t=T$, $r_T(s,1)-\lambda_T<r_T(s,1)-T*\max_{s,a,t}\subr_t(s,a)\leq 0$. On the other hand $r_T(s,0)\geq 0$ as all rewards are non-negative by the setting of our original MDP. Hence it is optimal to choose $a=0$. When $t<T$, $r_t(s,1)-\lambda_t+\sum_{s'\in \substates}P^a(s,s')V^{\lambdav}(s',t+1)< \max_{s,a,t}\subr_t(s,a)-T*\max_{s,a,t}\subr_t(s,a)+(T-t)*\max_{s,a,t}\subr_t(s,a) \leq 0 \leq r_t(s,0)$. Hence it is also optimal to choose $a=0$. Therefore $\beta_t(s)\leq T*\max_{s,a,t}\subr_t(s,a)$ for all $s,t$.

\section{A result that justifies using bisection}\label{ap:bisect}
\begin{lemma}\label{th:index}
If there exists an optimal policy $\subp$ that takes action $\subaction=1$ in state $\substate\in\substates$ at time $t$ for a sub-MDP $(\substates,\subactions,r,\subpr^{\cdot})$, and satisfies $P^{\subp}[S_t=s]>0$, then $\subaction=1$ is strictly optimal in state $\substate$ at time $t$ under a modified sub-MDP $(\substates,\subactions,r',\subpr^{\cdot})$ with $r'_t(\substate,1)>r_t(\substate,1)$, and $r'_t$ equals $r_t$ otherwise.
\end{lemma}
\begin{proof}{Proof of Lemma \ref{th:index}}
We prove by contradiction. Let $V(\pi,r)$ denote the total expected reward obtained by policy $\pi$ with reward function $r$. Assume that there exists an optimal policy $\subp'$ for sub-MDP $(\substates,\subactions,r'(\cdot,\cdot),\subpr^{\cdot}(\cdot,\cdot))$ such that $\subp'(s,0,t)=1$. Since neither $r_t(s,1)$ nor $r'_t(s,1)$ contributes to the total expected reward, $V(\pi',r)=V(\pi',r')$. Let $\subp$ be an optimal policy for $(\substates,\subactions,r(\cdot,\cdot),\subpr^{\cdot}(\cdot,\cdot))$. Then we have $V(\pi,r)\geq V(\pi',r)$. 
On the other hand, $V(\pi,r')$ is greater than $V(\pi,r)$ by $(r'_t(s,1)-r_t(s,1))\mathbb{P}^{\subp}[\substater_t=s]>0$. Hence we get that $V(\subp,r')>V(\subp,r)\geq V(\subp',r)=V(\subp',r')$ contradicting that $\subp'$ is an optimal policy of sub-MDP $(\substates,\subactions,r'(\cdot,\cdot),\subpr^{\cdot}(\cdot,\cdot))$. $\square$
\end{proof}
\section{}\label{ap:eq}
\begin{proof}{Proof of Lemma \ref{th:eq}}
To prove $(1)$, when $\beta_t(s)>\lambda^*_t$, by definition of the index in \eqref{df:index}, there exists an $\epsilon>0$ such that there is a $\pi\in\Pi^{*}(\lambdav^*[\lambda_t^*+\epsilon,t])$ and $\pi(s,1,t)=1$. Recall how we construct set $\Pi^*(\lambdav)$ in Section \ref{subsec:pi}, the value function $V^{\lambdav^*[\lambda^*_t+\epsilon,t]}$ corresponding to sub-MDP $Q(\lambdav^*[\lambda^*_t+\epsilon,t])$ has to satisfy
\begin{equation*}
r_t(s,1)-\lambda^*_t-\epsilon+\sum_{s'\in\substates}V^{\lambdav^*[\lambda^*_t+\epsilon,t]}(s',t+1)P^1(s,s')\geq r_t(\substate,0)+\sum_{s'\in\substates}V^{\lambdav^*[\lambda^*_t+\epsilon,t]}(s',t+1)P^0(s,s').
\end{equation*}
Since $\lambdav^*[\lambda^*_t+\epsilon,t]$ and $\lambdav^*$ share the same elements from the $(t+1)^{th}$ position onwards, $V^{\lambdav^*[\lambda^*_t+\epsilon,t]}(s,t') = V^{\lambdav^*}(s,t')$, for all $\substate\in\substates$ and $t'\geq t+1$. Hence
\begin{equation}\label{fp}
r_t(s,1)-\lambda^*_t+\sum_{s'\in\substates}V^{\lambdav^*}(s',t+1)P^1(s,s')> r_t(\substate,0)+\sum_{s'\in\substates}V^{\lambdav^*}(s',t+1)P^0(s,s').
\end{equation}
Next we consider two separate cases: 1) State $s$ is visited with positive probability under $\subp^{**}$, that is, $P^{\subp^{**}}(\substater_t=\substate)>0$; 2) State $s$ is visited with zero probability, i.e., $P^{\subp^{**}}(\substater_t=\substate)=0$. If 1) $P^{\subp^{**}}(\substater_t=\substate)>0$, since $\subp^{**}$ is an optimal policy for the unconstrained sub-MDP $Q(\lambdav^*)$ in \eqref{dpx}, and $a=1$ attains $\max\{r_t(s,a)-a\lambda_t^*+\sum_{s'\in\substates}P^{a}(s,s')V^{\lambdav^*}(s',t+1)\}$ alone, hence $\subp^{**}(s,1,t)=1$. If 2) $P^{\subp^{**}}(\substater_t=\substate)=0$, we get $\subp^{**}(s,1,t)=1$ directly from the construction of $\subp^{**}$ in \eqref{df:rho}. 

Statement (2) can be proven using a similar argument. We therefore skip the proof to avoid redundancy. $\square$
\end{proof}
\section{}\label{ap:p1}
\begin{proof}{Proof of Lemma \ref{th:p1}}
To prove $(1)$, suppose, for the sake of contradiction, that $\subp^{**}(s,1,t) < 1$. By Lemma \ref{th:eq}, we have $\beta_t(s) \leq \lambda^*_t$. Therefore $U_t(s)\cup V_t(s)$ forms a superset to the set of states with indices of at least $\lambda_t^*$. We also know that $\subp^{**}$ takes active action with probability $\alpha$ at time $t$. Hence we can write $\alpha$ as the sum of the probabilities of taking the active action in all states $s'$ with $\subp^{**}(s',1,t)=1$ and the probabilities of taking the active action in all states $s'$ with $0<\subp^{**}(s',1,t)<1$:
\begin{align}
\alpha &= \sum_{s'\in\{s'':\subp^{**}(s'',1,t)=1\}}P_t(s') * 1 + \sum_{s'\in\{s'':0<\subp^{**}(s'',1,t)<1\}}P_t(s') * \subp^{**}(s',1,t) \nonumber\\
& < \sum_{s'\in\{s'':\subp^{**}(s'',1,t)=1\}}P_t(s') + \sum_{s'\in\{s'':0<\subp^{**}(s'',1,t)<1\}}P_t(s').\label{greaterp}
\end{align}
Taking the contrapositives of both statements in Lemma \ref{th:eq}, we get if $0<\subp^{**}(s',1,t)<1$ then $\beta_t(s')=\lambda^*_t$. Hence
\begin{align*}
\eqref{greaterp} = \sum_{s'\in\{s'':\beta_t(s'')\geq 1\}}P_t(s') \leq \sum_{s'\in U_t(s)\cup V_t(s)}P_t(s')\leq \alpha
\end{align*}
We get $\alpha < \alpha$, which is a contradiction, as desired. 

To prove $(2)$, we again use contradiction. Assume $\subp^{**}(s,1,t)>0$; by the contrapositive of the second statement of Lemma \ref{th:eq} we know $\beta_t(s)\geq \lambda^*_t$. Then $U_{t+1}(s)$ is a subset of $\{s':\beta_t(s')>\lambda^*_t\}$, which in turn is a subset of $\{s':\subp^{**}(s',1,t)=1\}$ by Lemma \ref{th:eq}. 
Hence by the fact that $\alpha= \sum_{s'\in\{s'':\subp^{**}(s'',1,t)=1\}}P_t(s') * 1 + \sum_{s'\in\{s'':0<\subp^{**}(s'',1,t)<1\}}P_t(s') * \subp^{**}(s',1,t)$, we must have either 1)
\[\alpha> \sum_{s'\in\{s'':\subp^{**}(s'',1,t)=1\}}P_t(s')\geq \sum_{s'\in U_{t}(s)}P_t(s')\]
when there exists some $s'\in\{s'':0<\subp^{**}(s'',1,t)<1\}$ such that $P_{t}(s')>0$, or 2)
\[\alpha\geq \sum_{s'\in\{s'':\subp^{**}(s'',1,t)=1\}}P_t(s')> \sum_{s'\in U_{t}(s)}P_t(s')\]
otherwise, as we must have that $\subp^{**}(s,1,t)=1$. 
In either case we get $\alpha>\sum_{s'\in U_{t}(s)=1}P_t(s')$, which again forms a contradiction.
$\square$
\end{proof}
\section{}\label{ap:bino}
\begin{lemma}\label{th:bino}
Let $X^{(k)}$ be a sequence of non-negative random variables such that $\lim_{k\rightarrow \infty}\frac{1}{k}X^{(k)}=\gamma$, a.s.. If $Y^{(k)}|X^{(k)}\sim Bin(X^{(k)},p)$, then $\lim_{k\rightarrow \infty}\frac{Y^{(k)}}{k} = \gamma p$, a.s..
\end{lemma}
\begin{proof}{Proof of Lemma \ref{th:bino}}
We consider two cases: 1) $X^{(k)}\rightarrow\infty$; 2) $X^{(K)}$ is bounded. When 1) $X^{(k)}\rightarrow\infty$, we have
\begin{align*}
&\lim_{k\rightarrow\infty}\frac{Y^{(k)}}{k}\\
=&\lim_{k\rightarrow\infty}\frac{Y^{(k)}}{X^{(k)}}\frac{X^{(k)}}{k}\\
=&\lim_{k\rightarrow\infty}\frac{Y^{(k)}}{X^{(k)}}\lim_{k\rightarrow\infty}\frac{X^{(k)}}{k}\\
	=&p*\gamma
\end{align*}
If 2) $X^{(K)}$ is bounded, then $\gamma = 0$. $Y^{(k)}$ is also bounded, so $\lim_{k\rightarrow \infty}=0$.
$\square$
\end{proof}
\section{}\label{ap:movein}
\begin{proof}{Proof of Lemma \ref{th:movein}}
For readability, we write out $f_s(\frac{\Nv_{t+1}}{K},\frac{\lfloor\alpha K\rfloor}{K})$ as follows:
\begin{align}
f_s(\frac{\Nv_{t+1}}{K},\frac{\lfloor\alpha K\rfloor}{K}) =& \mathbbm{1}(\underbrace{[\frac{\lfloor\alpha K\rfloor}{K}-\sum_{s'\in U_{t+1}(s)}\frac{N_{t+1}(s')}{K}]^+ - \sum_{s'\in V_{t+1}(s)}\frac{N_{t+1}(s')}{K}}_{f_1(\frac{\Nv_{t+1}}{K},\frac{\lfloor\alpha K\rfloor}{K})}\geq 0)
*\frac{N_{t+1}(s)}{K} \nonumber\\
&+\mathbbm{1}([\frac{\lfloor\alpha K\rfloor}{K}-\sum_{s'\in U_{t+1}(s)}\frac{N_{t+1}(s')}{K}]^+ -\sum_{s'\in V_{t+1}(s)}\frac{N_{t+1}(s')}{K}<0)*\nonumber\\
&\hspace{4mm}\underbrace{\mathbbm{1}(\frac{\lfloor\alpha K\rfloor}{K}-\sum_{s'\in U_{t+1}(s)}\frac{N_{t+1}(s')}{K}>0)b_s(\frac{\Nv_{t+1}}{K},\frac{\lfloor\alpha K\rfloor}{K})}_{f_2(\frac{\Nv_{t+1}}{K},\frac{\lfloor\alpha K\rfloor}{K})}.
\end{align}
The convergence is based on that $\frac{\Nv_{t+1}}{K} \rightarrow P_{t+1}(s)$, a.s., by the induction step, and $\frac{\lfloor\alpha K\rfloor}{K}\rightarrow \alpha$.
We discuss by three cases:
\begin{itemize}
\item When $f_1(\Pv_{t+1}(s),\alpha)>0$,\\
by continuity of $f_1$ and the fact that $\frac{\Nv_{t+1}}{K} \rightarrow P_{t+1}(s)$, a.s., for almost any realization $\omega$ of the sequence $\{\frac{\Nv_{t+1}}{K}\}_K$, we can find a constant $K_0(\omega)$ such that $\forall K \geq K_0(\omega)$, $|f_1(\frac{\Nv_{t+1}(\omega)}{K},\frac{\lfloor\alpha K\rfloor}{K})-f_1(P_{t+1}(s),\alpha)|\leq \frac{f_1(P_{t+1}(s),\alpha)}{2}$.  
Hence $f_1(\frac{\Nv_{t+1}(\omega)}{K},\frac{\lfloor\alpha K\rfloor}{K}) > 0$, and $f(\frac{\Nv_{t+1}(\omega)}{K},\frac{\lfloor\alpha K\rfloor}{K}) = \frac{N_{t+1}(s)(\omega)}{K}$. As $K\rightarrow\infty$, $f_s(\frac{\Nv_{t+1}(\omega)}{K},\frac{\lfloor\alpha K\rfloor}{K})$ goes to $P_{t+1}(s)$. Since this holds true for almost every $\omega$, we have $f_s(\frac{\Nv_{t+1}}{K},\frac{\lfloor\alpha K\rfloor}{K})$ goes to $P_{t+1}(s) = f_s(\Pv_{t+1}(s),\alpha)$ almost surely. 
\item When $f_1(P_{t+1}(s),\alpha)<0$,\\
We first claim: for $K$ large enough, $H_1,H_2\in\{0,1\}$. To justify, first we look at the value of $b_s(\frac{\Nv_{t+1}}{K},\frac{\lfloor\alpha K\rfloor}{K},\frac{H_1}{K})$ when $\sum_{s'\in V_{t+1}(s)}\rho(s',1,t+1)>0$.  
As $K$ tends to infinity, the first min function in $b_s(\frac{\Nv_{t+1}}{K},\frac{\lfloor\alpha K\rfloor}{K})$, that is,  $$\frac{1}{K}\Big\lfloor (\lfloor\alpha K\rfloor-\sum_{s'\in U_{t+1}(s)}N_{t+1}(s'))\frac{\rho(s,1,t+1)}{\sum_{s'\in V_{t+1}(s)}\rho(s',1,t+1)}\Big\rfloor,$$ 
goes to $$ (\alpha -\sum_{s'\in U_{t+1}(s)}P_{t+1}(s'))\frac{\rho(s,1,t+1)}{\sum_{s'\in V_{t+1}(s)}\rho(s',1,t+1)},$$
 while the second term $\frac{N_{t+1}(s)}{K}$ goes to $P_{t+1}(s)$. 
 By assumption that $[\alpha -\sum_{s'\in U_{t+1}(s)}P_{t+1}(s')]^+ < \sum_{s'\in V_{t+1}(s)}P_{t+1}(s')$, we have
\begin{align*}
P_{t+1}(s) & \geq P_{t+1}(s)\pi^{**}(s,1,t+1)\\
&\geq (\alpha -\sum_{s'\in U_{t+1}(s)}P_{t+1}(s'))^+\frac{P_{t+1}(s)\pi^{**}(s,1,t+1)}{\sum_{s'\in V_{t+1}(s)}P_{t+1}(s')\pi^{**}(s',1,t+1)}\\
&=(\alpha -\sum_{s'\in U_{t+1}(s)}P_{t+1}(s'))^+\frac{\rho(s,1,t+1)}{\sum_{s'\in V_{t+1}(s)}\rho(s',1,t+1)},
\end{align*}
for all $s\in V_{t+1}(s)$. Therefore when $K$ is sufficiently large, we get
$$\frac{1}{K}\Big\lfloor (\lfloor\alpha K\rfloor-\sum_{s'\in U_{t+1}(s)}N_{t+1}(s'))^+\frac{\rho(s,1,t+1)}{\sum_{s'\in V_{t+1}(s)}\rho(s',1,t+1)}\Big\rfloor\leq\frac{N_{t+1}(s)}{K},
$$ 
for all $s\in V_{t+1}(s)$, and this satisfies the assumption in Remark \ref{remark:rounding}. Therefore we know that for sufficiently large $K$, $H_1$ takes value in $\{0,1\}$. 
When $\sum_{s'\in V_{t+1}(s)}\rho(s',1,t+1)=0$, we always have $[\lfloor\alpha K\rfloor-\sum_{s'\in U_{t+1}(s)}N_{t+1}(s')]^+*\frac{N_{t+1}(s)}{\sum_{s'\in V_{t+1}(s)}N_{t+1}(s')} \leq N_{t+1}(s)$,for all $s\in V_{t+1}(s)$, which satisfies the assumption in Remark \ref{remark:rounding}, hence $H_2\in\{0,1\}$. 

The rest of the argument to prove that $f_s(\frac{\Nv_{t+1}}{K},\frac{\lfloor\alpha K\rfloor}{K})$ goes to $ f_s(\Pv_{t+1}(s),\alpha)$ almost surely follows the same as the first case.
\item When $f_1(\Pv_{t+1},\alpha)=0$,\\
we claim that $P_{t+1}(s) = f_2(\Pv_{t+1},\alpha)$. To justify, we consider two cases:
\begin{enumerate}
\item $\alpha-\sum_{s'\in U_{t+1}(s)}P_{t+1}(s')\leq 0$.\\ Then we have $\sum_{s'\in V_{t+1}(s)}P_{t+1}(s') = 0$, which implies $P_{t+1}(s)=0$. 
$f_2(\Pv_{t+1},\alpha)$ is also zero since $\mathbbm{1}(\alpha -\sum_{s'\in V_{t+1}(s)}P_{t+1}(s')>0)=0$.
\item $\alpha-\sum_{s'\in U_{t+1}(s)}P_{t+1}(s')>0$.\\ When $\sum_{s'\in V_{t+1}(s)}\rho(s',1,t+1)=0$, the value of the second term is $P_{t+1}(s)$, since $(\alpha -\sum_{s'\in U_{t+1}(s)}P_{t+1}(s'))^+$ cancels with the denominator $\sum_{s'\in V_{t+1}(s)}P_{t+1}(s')$ by assumption. When $\sum_{s'\in V_{t+1}(s)}\rho(s',1,t+1)>0$, since we have $\alpha-\sum_{s'\in U_{t+1}(s)}P_{t+1}(s') - \sum_{s'\in V_{t+1}(s)}P_{t+1}(s') = 0$, by Lemma \ref{th:p1}, we have $\subp^{**}(s',a,t)=1$. Hence the term associated with the second indicator is again $P_{t+1}(s)$, which is the same as the term associted with the first indicator. (Recall $\rho(s,1,t)=\subp^{**}(s,1,t)*P_{t+1}(s)$.) 
 \end{enumerate}
 Therefore we have shown our claim.
 
 For any realization $\omega$ of $\{\frac{\Nv_{t+1}}{K}\}_K$, we define sequences $sq_1(\omega) = \{\frac{\Nv_{t+1}(\omega)}{K}:K=1,2,..., f_1(\frac{\Nv_{t+1}(\omega)}{K},\frac{\lfloor\alpha K\rfloor}{K})\geq 0\}$, and $sq_2(\omega) = \{\frac{\Nv_{t+1}(\omega)}{K}:K=1,2,..., f_1(\frac{\Nv_{t+1}(\omega)}{K},\frac{\lfloor\alpha K\rfloor}{K})< 0\}$.
  Then we have $\{f_2(\frac{\Nv_{t+1}(\omega)}{K},\frac{\lfloor\alpha K\rfloor}{K}):	\frac{\Nv_{t+1}(\omega)}{K}\in sq_2(\omega)\}$ and $sq_1(\omega)$ converge to the same value. Hence $f_s(\frac{\Nv_{t+1}(\omega)}{K},\frac{\lfloor\alpha K\rfloor}{K})$ goes to $P_{t+1}(s)$, and subsequently $f_s(\frac{\Nv_{t+1}}{K},\frac{\lfloor\alpha K\rfloor}{K})$ goes to $P_{t+1}(s) = f_s(\Pv_{t+1}(s),\alpha)$ almost surely.   $\square$
\end{itemize}
\end{proof}

\section{}\label{ap:equiv}
\begin{proof}{Proof of Lemma \ref{th:equiv}}
We divide our discussion into two main cases.
\begin{itemize}
\item 
When
\begin{equation}\label{eq:equiv_case1}
[\alpha-\sum_{s'\in U_t(s)}P_t(s')]^+ - \sum_{s'\in V_t(s)}P_t(s')\geq 0,
\end{equation} 
if $\alpha-\sum_{s'\in U_t(s)}P_t(s')\leq 0$, then $g_s(\Pv_t)=0$ in both cases. Moreover, \eqref{eq:equiv_case1} implies $P_t(s)=0$ by Lemma 4, hence $f_s(\Pv_t,\alpha)=0$.

If $\alpha-\sum_{s'\in U_t(s)}P_t(s')> 0$, \eqref{eq:equiv_case1} implies that $P_t(s)=1$ by Lemma 4.
$$\alpha-\sum_{s'\in U_t(s)}P_t(s')\geq \sum_{s'\in V_t(s)}P_t(s') = \sum_{s'\in V_t(s)}\rho(s',1,t) \geq P_t(s).$$
Hence $P_t(s)$ attains the minimum in both cases of $g_s(\Pv_t)$, and $f_s(\Pv_t,\alpha)=P_t(s) = g_s(\Pv_t)$.
\item When 
\begin{equation}\label{eq:equiv_case2}
[\alpha-\sum_{s'\in U_t(s)}P_t(s')]^+ - \sum_{s'\in V_t(s)}P_t(s') < 0,
\end{equation} 
if $\alpha-\sum_{s'\in U_t(s)}P_t(s')\leq 0$, then both $f_s(\Pv_t,\alpha)$ and $ g_s(\Pv_t)$ are zero.

If $\alpha-\sum_{s'\in U_t(s)}P_t(s')>0$, $f_s(\Pv_t,\alpha)$ is either $\min\{(\alpha-\sum_{s'\in U_t(s)}P_t(s'))\frac{\rho(s,1,t)}{\sum_{s'\in V_t(s)\rho(s',1,t)}},P_t(s)\}$
 or $\min\{(\alpha-\sum_{s'\in U_t(s)}P_t(s'))\frac{P_t(s)}{\sum_{s'\in V_t(s)P_t(s')}},P_t(s)\}$ 
 depending on whether $\sum_{s'\in V_t(s)}\rho(s',1,t)$ is greater than or equal to zero, which matches exactly the two cases in $g_s(\Pv_t)$.
\end{itemize}
\end{proof}
\section{}\label{ap:ppi}
\begin{proof}{Proof of Lemma \ref{th:ppi}}
We first consider the case when $\sum_{s'\in V_{t}(s)}\rho(s',1,t)>0$. This can be further divided into two sub-cases,
\begin{itemize}
\item Case 1: When $g_s(\Pv_t) = P_t(s)$,\\
if $\alpha-\sum_{s'\in U_t(s)}P_t(s')\leq 0$, by Lemma \ref{th:p1}, $\pi^{**}(s,1,t)=0$. Since $P_t(s)$ attains the minimum in this case, $P_t(s)=0$. Hence $g_s(\Pv_t)=0=P_t(s)\pi^{**}(s,1,t)$.

If $\alpha-\sum_{s'\in U_t(s)}P_t(s')> 0$, this leads to two cases by Lemma \ref{th:p1}: $\pi^{**}(s,1,t)=1$ or $0<\pi^{**}(s,1,t)<1$. If it is the latter, we have the second term in $g_s(\Pv_t)$ becomes $\rho(s,1,t)$, since $\alpha-\sum_{s'\in U_t(s)}P_t(s')$ cancels with the denominator. $\rho(s,1,t)=P_t(s)\pi^{**}(s,1,t)<P_t(s)$, which contradicts that $P_t(s)$ attains the minimum. Hence $\pi^{**}(s,1,t) = 1$. We have $g_s(\Pv_t) = P_t(s) = P_t(s)*\pi^{**}(s,1,t)$.
\item Case 2: When $g_s(\Pv_t) = [\alpha-\sum_{s'\in U_t(s)}P_t(s')]^+\frac{\rho(s,1,t)}{\sum_{s'\in V_t(s)\rho(s',1,t)}}$,\\
if $\alpha-\sum_{s'\in U_t(s)}P_t(s') \leq 0$, again by Lemma \ref{th:p1}, $\pi^{**}(s,1,t)=0$, $g_s(\Pv_t) = 0 = P_t(s)\pi^{**}(s,1,t)$.

If $\alpha-\sum_{s'\in U_t(s)}P_t(s') > 0$, again we have two cases: $\pi^{**}(s,1,t)=1$ or $0<\pi^{**}(s,1,t)<1$. If it is the latter, we have $\alpha-\sum_{s'\in U_t(s)}P_t(s')=\sum_{s'\in V_t(s)}P_t(s')$. 
If it is the former, by assumption $[\alpha-\sum_{s'\in U_t(s)}P_t(s')]\frac{P_t(s)}{\sum_{s'\in V_t(s)P_t(s')}}$ attains the minimum, $\alpha-\sum_{s'\in U_t(s)}P_t(s') \leq \sum_{s'\in V_t(s)}P_t(s')$. Hence $\alpha-\sum_{s'\in U_t(s)}P_t(s')$ can only be equal to $\sum_{s'\in V_t(s)}P_t(s')$. Subsequently for both cases  $\pi^{**}(s,1,t)=1$ and $0<\pi^{**}(s,1,t)<1$, we have $g_s(\Pv_t) = \rho(s,1,t)=P_t(s)\pi^{**}(s,1,t)$.
\end{itemize}

Now we look at the case when $\sum_{s'\in V_t(s)}\rho(s',1,t)=0$. We have either 1) $P_t(s')=0$ for all $s'\in V_t(s)$ 
2) $\pi^{**}(s',1,t)=0$ for all $s'\in V_t(s)$. 

When $\pi^{**}(s',1,t)=0$ for all $s'\in V_t(s)$, $\alpha-\sum_{s'\in U_t(s)}P_t(s')\leq 0$
hence $g_s(\Pv_t) = 0 = P_t(s)\pi^{**}(s,1,t)$.  
\end{proof}

\section*{Acknowledgements}
We thank NSF and AFOSR for funding this project.

\bibliographystyle{imsart-nameyear.bst}
\bibliography{myRef.bib} 
\end{document}